\documentclass{amsart}
\usepackage{amsmath,amsfonts,euscript,amscd,amsthm,amssymb,upref,graphics}
\usepackage{mathrsfs}

\usepackage[all]{xy}

\usepackage{latexsym}
\usepackage{amsfonts,amssymb} 
\usepackage{graphicx}
\usepackage{tikz}
\usepackage{caption}
\usepackage{subcaption}
\usepackage{enumitem}

\newcommand{\field}[1]{\mathbb{#1}}
\newcommand{\A}{\field{A}}
\newcommand{\C}{\field{C}}

\newcommand{\G}{\field{G}}

\newcommand{\N}{\field{N}}
\newcommand{\PP}{\field{P}}
\newcommand{\Q}{\field{Q}}
\newcommand{\R}{\field{R}}

\newcommand{\Z}{\field{Z}}

\newcommand{\krn}{{\rm ker}\,}

\theoremstyle{plain}

\newtheorem{theorem}{Theorem}[section]
\newtheorem{proposition}[theorem]{Proposition}
\newtheorem{lemma}[theorem]{Lemma}
\newtheorem{corollary}[theorem]{Corollary}
\newtheorem{definition}[theorem]{Definition}
\newtheorem{remark}[theorem]{Remark}

\newtheorem{example}[theorem]{Example}
\newtheorem{question}[theorem]{Question}

\theoremstyle{definition}

\theoremstyle{remark}

\begin{document}

% Redefine "plain" pagestyle
\makeatletter	   % `@' is now a normal "letter' for LaTeX
\makeatother     % `@' is restored as a "non-letter" character

\title{Actions of $SL_2(k)$ on affine $k$-domains and fundamental pairs}
\author{Gene Freudenburg}
\date{\today}
\maketitle

%%%%%%%%%%%%%%%%%%%%%%%%%%%%%%%%%%%%%%%%%%%%%%%%%%%%%%%%%%%%%%%%%%%%%%%%%%%%%%
\begin{abstract} Working over a field $k$ of characteristic zero, this paper studies algebraic actions of $SL_2(k)$ on affine $k$-domains by defining and investigating fundamental pairs of derivations. There are three main results:
(1) The Structure Theorem for Fundamental Derivations ({\it Theorem\,\,\ref{main1}}) describes the kernel of a fundamental derivation, together with its degree modules and image ideals. 
(2) The Classification Theorem ({\it Theorem\,\ref{classification}}) lists all normal affine $SL_2(k)$-surfaces with trivial units, generalizing the classification given by Gizatullin and Popov for complex $SL_2(\C )$-surfaces. 
(3) The Extension Theorem ({\it Theorem\,\ref{main3}}) describes the extension of a fundamental derivation of a $k$-domain $B$ to $B[t]$ by an invariant function. 
The Classification Theorem is used to describe three-dimensional UFDs which admit a certain kind of $SL_2(k)$ action ({\it Theorem\,\ref{threefold}}). 
This description is used to show that any $SL_2(k)$-action on $\A_k^3$ is linearlizable, which was proved by Kraft and Popov in the case $k$ is algebraically closed. 
This description is also used, together with Panyushev's theorem on linearization of $SL_2(k)$-actions on $\A_k^4$, to show a cancelation property for threefolds $X$: If $k$ is algebraically closed, $X\times\A_k^1\cong\A_k^4$ and $X$ admits a nontrivial action of $SL_2(k)$, then $X\cong\A_k^3$ ({\it Theorem\,\ref{cancellation}}). 
The Extension Theorem is used to investigate free $\G_a$-actions on $\A_k^n$ of the type first constructed by Winkelmann.
\end{abstract}

%%%%%%%%%%%%%%%%%%%%%%%%%%%%%%%%%%%%%%%%%%%%%%%%%%%%%%%%%%%%%%%%%%%%%%%%%%%%%%

\section{Introduction}  Let $k$ be a field of characteristic zero. In the introduction of their paper \cite{Arzhantsev.Liendo.12}, Arzhantsev and Liendo write:
\begin{quotation}
A regular $SL_2$-action on an affine variety $X$ is uniquely defined by an 
$\mathfrak{sl}_2$-triple $\{ \partial ,\partial_+,\partial_-\}$ of derivations of the algebra $k[X]$, where the $\partial_{\pm}$ are locally nilpotent, $\partial =[\partial_+,\partial_-]$ is semisimple, and
$[\partial,\partial_{\pm}]=\pm2\partial_{\pm}$. 
\end{quotation}
In Proposition 2.1 they give the following criterion for the existence of an $SL_2$-action which, according to the authors, is well-known.
Their ground field is algebraically closed, but the proof they give for this result is valid over any field of characteristic zero.

\begin{proposition} A nontrivial $SL_2(k)$-action on an affine $k$-variety $X={\rm Spec}(A)$ is equivalent to a $\Z$-grading of $A$ with infinitesimal generator $\delta$ and homogeneous locally nilpotent derivations $\delta_+$ and $\delta_-$ of $A$ such that $\deg\delta_+=2$, $\deg\delta_-=-2$ and 
$[\delta_+,\delta_-]=\delta$. 
\end{proposition} 
This paper investigates pairs of locally nilpotent derivations $(\delta_+,\delta_-)$ described in this proposition. We call $\delta_+$ and $\delta_-$ fundamental derivations 
and  $(\delta_+,\delta_-)$ a fundamental pair; see {\it Section 3}. 
There are three main results. The Structure Theorem for Fundamental Derivations ({\it Theorem\,\,\ref{main1}}) gives a detailed description of the kernel of a fundamental derivation, 
together with its degree modules and image ideals. 
The Classification Theorem ({\it Theorem\,\ref{classification}}) gives all normal affine two-dimensional $k$-algebras with trivial units which admit a fundamental pair, equivalently, 
normal affine $SL_2(k)$-surfaces over $k$ with trivial units. This generalizes the classification given by Gizatullin and Popov for complex quasihomogeneous surfaces \cite{Gizatullin.71c,Popov.73a}; 
see also \cite{Huckleberry.86}. 
The Extension Theorem ({\it Theorem\,\ref{main3}}) describes the extension of a fundamental derivation on a $k$-domain $B$ to $B[t]$ by an $SL_2(k)$-invariant function, 
using the degree modules of the derivation on $B$ to give an explicit set of generators for the kernel of the extended derivation on $B[t]$. 

In general, calculation of image ideals and degree modules of a locally nilpotent derivation of an affine ring is an algorithmically intensive process, even when generators for the kernel are known.
In contrast, the Structure Theorem shows that image ideals and degree modules for a fundamental derivation are completely determined by the induced $\Z$-grading of the ring. 
The Structure Theorem and Classification Theorem combine to give {\it Proposition\,\ref{threefold}}, which characterizes UFDs with a certain type of fundamental pair. 
This yields a proof that every algebraic $SL_2(k)$-action on $\A_k^3$ is linearizable 
({\it Theorem\,\,\ref{Kraft-Popov}}). Kraft and Popov \cite{Kraft.Popov.85} showed that, when $k$ is algebraically closed, the action of any connected semisimple group on $\A_k^3$ is linearizable by first showing this for the group $SL_2(k)$. 
{\it Theorem\,\,\ref{Kraft-Popov}} thus generalizes the result of Kraft and Popov for $SL_2(k)$. 
Panyushev \cite{Panyushev.84} showed that, when $k$ is algebraically closed, the action of any connected semisimple algebraic group on $\A_k^4$ is linearizable, 
remarking: ``The case of the group $SL_2$ is very interesting and complicated'' (p.171). 
{\it Proposition\,\ref{threefold}} is further used in conjunction with 
Panyushev's theorem, the Epimorphism Theorem and the cancelation theorems for curves and surfaces to obtain the 
following cancelation property ({\it Theorem\,\ref{cancellation}}):
\begin{quotation}
{\it Let $k$ be an algebraically closed field of characteristic zero. 
Assume that $X\times\A_k^1\cong\A_k^4$ for some threefold $X$. If $X$ admits a nontrivial action of $SL_2(k)$, then $X\cong\A_k^3$.}
\end{quotation}

In \cite{Winkelmann.90} Winkelmann constructed 
a free $\C^+$-action on $\C^4$ with singular algebraic quotient, and a locally trivial $\C^+$-action on $\C^5$ with smooth algebraic quotient which is not globally trivial. 
In \cite{Finston.Jaradat.17} Finston and Jaradat used similar methods to construct a locally trivial $\C^+$-action on $\C^5$ with singular algebraic quotient.
We generalize these constructions by recognizing and exploiting the role of the underlying $SL_2(\C )$-module, showing in particular that the quotient morphism of such an action cannot be surjective ({\it Theorem\,\ref{main3}}). The Extension Theorem gives a simple way to 
confirm the results of Finston and Jaradat, who used the van den Essen algorithm (and {\it Singular}) to find generators for their ring of invariants.
This was the first example of a locally trivial $\C^+$-action on an affine space having a singular algebraic quotient. 
We give a simpler example of a locally trivial $\C^+$-action on $\C^5$ with singular algebraic quotient ({\it Section\,\ref{singular}}). 

This paper presents topics in reverse chronological order relative to the research it represents. The concrete examples considered in {\it Section\,\ref{examples}} were the starting point. In particular, the author 
used Winkelmann's method to construct a certain locally trivial $\C^+$-action on $\C^5$, unaware at the time that Finston and Jaradat had already considered exactly the same example.
The Extension Theorem resulted from efforts to find its ring of invariants. This exercise led to the Structure Theorem, first for the basic linear derivations, and eventually to its current form 
 in {\it Section\,\ref{structure}}. 

%%%%%%%%%%%%%%%%%%%%%%%%%%%%%%%%%%%%%%%%%%%%%%%%%%%%%%%%%%%%%%%%%%%%%%%%%%%%%%

\section{Preliminaries}\label{prelims}
\subsection{$k$-Domains and Derivations}
Throughout, $k$ denotes a field of characteristic zero, $\C$ the field of complex numbers, $\Z$ the ring of integers, and $\N$ the semi-ring of non-negative integers. 
A {\bf $k$-domain} is an integral domain $B$ containing $k$; its field of fractions is ${\rm frac}(B)$, and its group of units is $B^*$. 
${\rm Aut}_k(B)$ is the group of $k$-algebra automorphisms of $B$ and ${\rm End}_k(B)$ is the ring of $k$-linear operators on $B$. 
The polynomial ring in $n$ variables over $B$ is $B^{[n]}$ and the field of fractions of $k^{[n]}$ is $k^{(n)}$, and if $k[x_1,\hdots ,x_n]=k^{[n]}$ then the corresponding system of partial derivatives is
$\partial_{x_1},\hdots ,\partial_{x_n}$. 
If $B$ is $k$-affine, then the dimension of $B$ over $k$ is $\dim_kB$. 
$\A_k^n$ is affine $n$-space over $k$. 

For any commutative ring $R$, ${\rm Der}(R)$ is the set of derivations of $R$, and for a subring $S\subset R$, ${\rm Der}_S(R)$ is the set of derivations with $DS=\{ 0\}$, called 
{\bf $S$-derivations}. We say that $D\in {\rm Der}(R)$ is {\bf reducible} if there exists $r\in R$ with $DR\subset rR\ne R$; otherwise, $D$ is {\bf irreducible}. For $D,E\in {\rm Der}(R)$, $[D,E]=DE-ED\in {\rm Der}(R)$ is the Lie product of $D$ and $E$, which gives ${\rm Der}(R)$ the structure of a Lie algebra.
Note that ${\rm Der}_S(R)$ is a Lie subalgebra.

$S$ is a {\bf retract} of $R$ if there is an ideal $I\subset R$ such that $R=S\oplus I$ as $S$-modules. 
If $R$ is an integral domain, then $I$ is necessarily a prime ideal. 

\subsection{Locally Nilpotent Derivations}\label{LNDs}

For a $k$-domain $B$, $D\in {\rm Der}(B)$ is {\bf locally nilpotent} if, to each $b\in B$ there is $n\in\N$ with $D^nb=0$. 
The set of locally nilpotent derivations of $B$ is ${\rm LND}(B)$. For a subring $S\subset B$, ${\rm LND}_S(B)={\rm LND}(B)\cap {\rm Der}_S(B)$. 

Let $D\in {\rm LND}(B)$ be given. Several basic definitions and properties follow, details of which can be found in \cite{Freudenburg.17}. 
\begin{enumerate}
\item Given nonzero $b\in B$, the least integer $n\in\N$ with $D^{n+1}b=0$ is the {\bf degree} of $b$ relative to $D$, denoted $\deg_D(b)$. This defines a degree function $\deg_D :B\to\N\cup\{ -\infty\}$. 
\item Let $A=\krn D$, the kernel of $D$ (also denoted $B^D$). Then $A$ is a factorially closed subring of $B$, being the set of elements of degree at most 0. Consequently, 
$B^*\subset A$, $A$ is a $k$-subalgebra of $B$ and ${\rm LND}(B)\subset {\rm Der}_k(B)$. However, ${\rm LND}(B)$ is not a Lie subalgebra of ${\rm Der}_k(B)$ since it is not a $k$-vector subspace, and it is not closed under Lie brackets.  
\item An ideal $I\subset B$ is {\bf $D$-invariant} if $DI\subset I$. Given $D'\in{\rm LND}(B)$, $I$ is {\bf $(D,D')$-invariant} if it is both $D$-invariant and $D'$-invariant. 
\item Any element $r\in B$ of degree one is a {\bf local slice} of $D$. If $a=Dr\in A$, then $a\ne 0$ and $B_a=A_a[r]\cong A_a^{[1]}$. Here, $B_a$ and $A_a$ denote localization at the set $\{a^n\, |\, n\in\N\}$. 
Moreover, $D$ extends to $D_a$ on $B_a$ and $D_a=a\frac{d}{dr}$. If $Dr=1$ then $r$ is a {\bf slice} for $D$ and $B=A[r]\cong A^{[1]}$. 
\item Let $r\in B$ be a local slice of $D$ and $a=Dr$. The induced {\bf Dixmier map} $\pi_r:B_a\to A_a$ is the surjective map of $k$-algebras defined by
\[
\pi_r(b)=\sum_{i\ge 0}{\textstyle\frac{(-1)^i}{i!}}D^ib\left({\textstyle\frac{r}{a}}\right)^i
\]
\item If $B$ is generated by $b_1,\hdots ,b_m$ over $k$, then $A$ is algebraic over $k[\pi_r(b_1),\hdots ,\pi_r(b_m)]$. This fact forms the basis of the van den Essen algorithm to calculate a generating set for $A$ when both $A$ and $B$ are finitely generated over $k$; see \cite{Freudenburg.17}, Section 8.1.
\item The {\bf degree modules} of $D$ are $\mathcal{F}_n:=\krn D^{n+1}=\{ b\in B\, |\, \deg_Db\le n\}$ for $n\ge 0$. 
Each $\mathcal{F}_n$ is an $A$-module and $B$ has the ascending $\N$-filtration $B=\bigcup_{i\ge 0}\mathcal{F}_i$. An algorithm to calculate degree modules is given in \cite{Freudenburg.17}, Section 8.6. 
\item The {\bf image ideals} of $D$ are $I_n=A\cap D^nB=D^n\mathcal{F}_n$ for $n\ge 0$. This gives a descending $\N$-filtration of $A$. The {\bf plinth ideal} of $A$ is $I_1=A\cap DB$. Note that the image ideals are ideals in $A$, not $B$. 
\item $D$ has the {\bf Freeness Property} if $D$ satisfies the following equivalent conditions. 
\subitem (i) $\mathcal{F}_n/\mathcal{F}_{n-1}$ is a free $A$-module for each $n\ge 1$.
\subitem (ii) $I_n$ is a principal ideal of $A$ for each $n\ge 1$.
\item The {\bf degree-$n$ transvectant} is the $A$-bilinear mapping $\mathcal{F}_n\times\mathcal{F}_n\to A$ given by:
\[
[f,g]_n^D=\sum_{i=0}^n(-1)^{n-i}D^ifD^{n-i}g
\]
\item Given nonzero $a\in A$, a {\bf $D$-cable rooted at $a$} is a sequence of nonzero elements $P_n\subset B$ such that $P_0\in ka$ and $DP_n\in kP_{n-1}$ for $n\ge 1$. 
If the sequence is finite, say $\{ P_0,\hdots ,P_d\}$, then $\{ P_n\}$ is a {\bf finite} $D$-cable of {\bf length} $d$. 
Basic properties of $D$-cables are laid out in \cite{Freudenburg.13}. 
\end{enumerate}

\begin{lemma}\label{plinth} Suppose that $B$ is an integral $k$-domain. Let $D\in {\rm LND}(B)$, $A=\krn D$ and $I_n=  A\cap D^nB$, $n\ge 0$. If $E\in {\rm Der}_k(B)$ 
and $[D,E]=0$, then $E(I_n)\subset I_n$ for all $n\ge 0$. 
\end{lemma}

\begin{proof} Given $f\in A$, $DE(f)=ED(f)=0$ implies $Ef\in A$. Given $n\ge 0$ and $a\in I_n$ let $b\in B$ be such that $D^nb=a$. Then 
$Ea=ED^n(b)=D^n(Eb)\in D^nB\cap A=I_n$.
\end{proof}

\begin{lemma}\label{princ-6} {\rm (\cite{Freudenburg.17}, Principle 6)}
Let $B$ be a commutative $k$-domain, $D\in {\rm LND}(B)$ and $A=\krn D$. Let $B'=B[t]\cong B^{[1]}$ and extend $D$ to $B'$ by $Dt=0$. 
Given $b\in B$ define $D'\in {\rm Der}_k(B')$ by $D'=D+b\frac{d}{dt}$ and let $A'=\krn D'$. 
\begin{itemize}
\item [{\bf (a)}] $D'\in {\rm LND}(B')$
\item [{\bf (b)}] $[D',\frac{d}{dt}]=[D,\frac{d}{dt}]=0$ and $\frac{d}{dt}$ restricts to $A'$.
\smallskip
\item [{\bf (c)}] If $b\in A$ then $[D,D']=0$.
\end{itemize}
\end{lemma} 

For any field $K$ and $K$-domain $B$, 
the {\bf Makar-Limanov invariant} $ML(B)$ of $B$ 
is the $K$-subalgebra consisting of elements invariant for any $\G_a$-action on $B$. 
In characteristic zero, $ML(B)$ is the intersection of all kernels of locally nilpotent derivations of $B$, and is therefore factorially closed (hence algebraically closed) in $B$ and contains $B^*$. 
So if $ML(B)=k$, then $k$ is algebraically closed in $B$ and $B^*=k^*$.

Following are two additional results that are needed. 

\begin{lemma}\label{Kolhatkar} {\rm (\cite{Kolhatkar.11}, Lemma 2.14)} Let $k$ be a field of characteristic zero, and let $B$ be a two-dimensional normal affine $k$-domain with $ML(B)=k$.
$\krn D\cong k^{[1]}$ for every nonzero $D\in {\rm LND}(B)$.
\end{lemma} 

\begin{theorem}\label{Miyanishi}  Let $k$ be a field of characteristic zero, $B=k^{[3]}$, $D\in {\rm LND}(B)$, $D\ne 0$, and $A=\krn D$. 
\begin{itemize}
\item [{\bf (a)}] $A\cong k^{[2]}$
\item [{\bf (b)}] The plinth ideal $A\cap DB$ is a principal ideal of $A$. 
\end{itemize}
\end{theorem}
Part (a) is known as Miyanishi's Theorem \cite{Miyanishi.85}. 
Miyanishi proved this result for the field $k=\C$. The general case is obtained from this using Kambayashi's classification of separable forms of the affine plane \cite{Kambayashi.75}. 
For part (b) see \cite{Freudenburg.17}, Theorem 5.12. 

\subsection{Group actions} Let $B$ be an affine $k$-domain. When an algebraic group $G$ over $k$ acts algebraically on $B$, the ring of invariants is $B^G$.
$\G_a$ denotes the additive group of $k$ and $\G_m$ denotes the multiplicative group $k^*$. In case $B\cong k^{[n]}$ the $G$-action is {\bf linearizable} 
if there exist $x_1,\hdots ,x_n\in B$ such that $B=k[x_1,\hdots ,x_n]$ and $G$ restricts to a linear action on the vector space $kx_1\oplus\cdots\oplus kx_n$. 
Similarly, an action of $G$ on affine space $\A_k^n$ is linearizable if the corresponding action on $k^{[n]}$ is linearizable. 

There is a bijective correspondence between ${\rm LND}(B)$ and the algebraic $\G_a$-actions on ${\rm Spec}(B)$, where $D\in {\rm LND}(B)$ defines the $\G_a$-action on $B$ (and ${\rm Spec}(B)$) by $\{ \exp (tD)\, |\, t\in k\}$. Under this correspondence,
the fixed point set of the action is defined by the ideal $(DB)$ generated by the image of the corresponding derivation $D$, and 
$B^D=B^{\G_a}$. In general, this invariant ring is not affine, but it is always quasi-affine \cite{Daigle.Freudenburg.99, Winkelmann.03}. 
In case $B^D$ is affine, the algebraic quotient ${\rm Spec}(B^D)$ equals the categorical quotient and the quotient morphism ${\rm Spec}(B)\to {\rm Spec}(B^D)$ is induced by the inclusion $B^D\to B$. 

Similarly, there is a bijective correspondence between the $\Z$-gradings of $B$ and the algebraic $\G_m$-actions on $B$ (and ${\rm Spec}(B)$).
If $B=\bigoplus_{n\in\Z}B_n$ is a $\Z$-grading then the induced action on $B_n$ is $\lambda (f)=\lambda^nf$, $\lambda\in k^*$, and this extends to all of $B$. Conversely, if $\G_m$ acts on $B$ then the induced $\Z$-grading is $B=\bigoplus_{n\in\Z}B_n$ where $B_n=\{ f\in B\, |\, \lambda (f)=\lambda^nf \,\,\forall \lambda\in k^*\}$. 

We refer to the following results. The first of these is well-known and elementary. A proof is provided for completeness. 
\begin{lemma} Let $K$ be a field. The group $SL_2(K)$ is generated by elements of the form
\[
\begin{pmatrix} 1& 0\cr t&1\end{pmatrix} \quad and \quad \begin{pmatrix} 1& s\cr 0&1\end{pmatrix} \,\, (s,t\in K).
\]
\end{lemma}
 \begin{proof} Let $G$ be the subgroup of $SL_2(K)$ generated by elements of the given form, and let $a,b,c,d$ be elements of $K$ with $ad-bc=1$. If $a\in K^*$, then
 \[
 \begin{pmatrix} a&0\cr 0&a^{-1}\end{pmatrix} = \begin{pmatrix} 1&0\cr a^{-1}&1\end{pmatrix}  \begin{pmatrix} 1&1-a\cr 0&1\end{pmatrix} \begin{pmatrix} 1&0\cr -1&1\end{pmatrix} \begin{pmatrix} 1&1-a^{-1}\cr 0&1\end{pmatrix}\in G
 \]
which implies:
 \[
 \begin{pmatrix}a&b\cr c&d \end{pmatrix}= 
 \begin{pmatrix} 1&0\cr ca^{-1}&1\end{pmatrix} \begin{pmatrix} a&0\cr 0&a^{-1}\end{pmatrix}\begin{pmatrix} 1&a^{-1}b\cr 0&1\cr \end{pmatrix}\in G
  \]
 If $a=0$ then $c\in K^*$ and:
 \[
  \begin{pmatrix} 0&-c^{-1}\cr c&d\end{pmatrix}
 = \begin{pmatrix} 1&0\cr c&1\end{pmatrix}  \begin{pmatrix} 1&-c^{-1}\cr 0&1\end{pmatrix} \begin{pmatrix} 1&0\cr c&1\end{pmatrix}\begin{pmatrix} 1&dc^{-1}\cr 0&1\end{pmatrix}\in G
  \]
 Therefore, $G=SL_2(K)$.
 \end{proof} 
 
 \begin{theorem} {\rm (Finiteness Theorem; see \cite{Freudenburg.17}, 6.1)} If $B$ is an affine $k$-domain and $G$ is a reductive group over $k$ which acts algebraically on $B$, then $B^G$ is affine. 
 \end{theorem}
 
\begin{theorem} {\rm (Kraft and Popov \cite{Kraft.Popov.85})} Let $G$ be a connected semisimple group over an algebraically closed field of $k$ characteristic zero. Any regular action of $G$ on $\A_k^3$ is equivalent to a linear one. 
\end{theorem}
The authors do not include the condition {\it algebraically closed} in their statement of the theorem, but this is an oversight, as they go on to say, ``We always work in the category of algebraic varieties over the field $\C$ of
complex numbers. Of course we could replace $\C$ by any other algebraically closed field of characteristic zero.''

\begin{theorem} {\rm (Panyushev \cite{Panyushev.84})} Let G be a connected semisimple algebraic group of biregular automorphisms of four-dimensional affine space $\A_k^4$ over an algebraically closed field $k$ of characteristic zero. 
The action of G on $\A_k^4$ is equivalent to a linear action.
\end{theorem} 

 The representations of $SL_2(k)$ on the vector spaces $k^{n+1}$, $n\ge 1$, are well-known. Given $n\ge 1$ define linear operators on $kx_0\oplus\cdots \oplus kx_n$ by:
\begin{equation}\label{down-op}
Dx_i=x_{i-1} \,\, (1\le i\le n)\,\, ,\,\, Dx_0=0 \quad {\rm and}\quad  Ux_i=(i+1)(n-i)x_{i+1}\,\, (0\le i\le n-1)\,\, , \,\, Ux_n=0
\end{equation}
$D$ is the {\bf down operator}, $U$ the {\bf up operator}, and $E=[D,U]$ the {\bf Euler operator}. 
Then $G_1=\{ \exp (tD)\, |\, t\in k\}\cong\G_a$ is a unipotent one-parameter subgroup of lower triangular matrices and $G_2=\{ \exp (sU)\, |\, s\in k\}\cong\G_a$ a unipotent one-parameter subgroup of upper-triangular matrices. 
Together they generate a copy of $SL_2(k)$ in $GL_{n+1}(k)$ and define an irreducible representation on $k^{n+1}$. We fix the following notation.
\begin{quotation}
Given $n\ge 0$, $V_n$ is the irreducible representation of $SL_2(k)$ on $k^{n+1}=kx_0\oplus\cdots\oplus kx_n$ defined by the up and down operators. 
\end{quotation}
Note that $V_n$ is
equivalent to the classical representation of $SL_2(k)$ on binary forms of degree $n$. The coefficients for $V_n$ are more natural from the point of view of locally nilpotent derivations than the classical representation. 

\subsection{Extending a theorem of Gupta}

In \cite{Gupta.14c}, Theorem A, N. Gupta gives the following.

\begin{theorem} Let $K$ be a field and $A$ a $K$-domain of the form
\[
A=K[X,Y,Z,T]/(X^mY-F(X,Z,T))
\]
where $F\in K[X,Z,T]\cong_KK^{[3]}$ and $m\ge 2$. Then $A\cong_KK^{[3]}$ if and only if $F(0,Z,T)$ is a variable in $K[Z,T]$. 
\end{theorem} 

Gupta's proof relies on the fact that $ML(A)=K[X]$. If $m=1$ then $ML(A)=K$ and this case was not treated in her paper. 
We extend Gupta's theorem to the case $m=1$ when the ground field is of characteristic zero. 
Our proof uses Kaliman's Fiber Theorem ({\it Theorem\,\ref{Kaliman}}), which was first proved by Kaliman for $k=\C$ \cite{Kaliman.02} and later extended to any field of characteristic zero by Daigle and Kaliman \cite{Daigle.Kaliman.09}, Theorem 4.2. 

\begin{theorem}\label{Kaliman} Let $B=k^{[3]}$ where $k$ is a field of characteristic zero. Given $f\in B$ suppose that the set 
\[
\{ \mathfrak{p}\in {\rm Spec}(k[f])\, |\, B\otimes_{k[f]}\kappa (\mathfrak{p})=\kappa (\mathfrak{p})^{[2]}\}
\]
is dense in ${\rm Spec}(B)$, where $\kappa (\mathfrak{p})$ is the residue field $R_{\mathfrak{p}}/\mathfrak{p}R_{\mathfrak{p}}$ at $\mathfrak{p}$.
Then $f$ is a variable of $B$.
\end{theorem}

\begin{theorem}\label{Gupta} Let $k$ be a field of characteristic zero and $A$ a $k$-domain of the form
\[
A=k[X,Y,Z,T]/(XY-F(X,Z,T))
\]
where $F\in k[X,Z,T]\cong_kk^{[3]}$. Then $A\cong_kk^{[3]}$ if and only if $F(0,Z,T)$ is a variable in $k[Z,T]$. 
\end{theorem} 

\begin{proof} Let $x\in A$ be the image of $X$ and let $F_0(Z,T)=F(0,Z,T)$. 

Assume that $A\cong k^{[3]}$. 
Given nonzero $c\in k$,
\[
A/(x-c)A\cong k[Y,Z,T]/(cY-F(c,Z,T))=k[\bar{Z},\bar{T}]\cong k^{[2]}
\]
where $\bar{Z},\bar{T}$ denote the images of $Z$ and $T$ in this quotient. 
By the Kaliman Fiber Theorem, $x$ is a variable of $A$, which implies that every fiber of $x$ is isomorphic $\A_k^2$.
In particular, if $R=k[Z,T]/(F_0)$ then the cancellation theorem for curves \cite{Abhyankar.Eakin.Heinzer.72} gives:
\[
A/xA=R[\bar{Y}]\cong k^{[2]} \implies R\cong k^{[1]}
\]
So $F_0$ defines a line in ${\rm Spec}(k[Z,T])=\A_k^2$. By the Epimorphism Theorem \cite{Abhyankar.Moh.75,Suzuki.74} it follows that $F_0$ is a variable of $k[Z,T]$.

Conversely, assume that $F_0$ is a variable of $k[Z,T]$ and let $G\in k[Z,T]$ be such that $k[Z,T]=k[F_0,G]$. Let $H\in k[X,Z,T]$ be such that $F=XH+F_0$.  
Then 
\[
k[X,Y,Z,T]=k[X,Y-H,F_0,G]=k[X,Y-H,F_0-X(Y-H),G]=k[X,Y-H,XY-F,G]
\]
which implies $XY-F$ is a variable and $A\cong k^{[3]}$. 
\end{proof}

%%%%%%%%%%%%%%%%%%%%%%%%%%%%%%%%%%%%%%%%%%%%%%%%%%%%%%%%%%%%%%%%%%%%%%%%%%%%%%

\section{The Structure theorem for fundamental derivations}\label{structure}
\subsection{Main definition}
The Lie algebra $\mathfrak{sl}_2(k)$ is generated by elements $X$ and $Y$ with relations:
\[
[X,[X,Y]]=-2X\,\, ,\,\, [Y,[X,Y]]=2Y
\]
With the criterion of Arzhantsev and Liendo in mind, our main objects of interest are pairs of locally nilpotent derivations which generate a subalgebra isomorphic to 
$\mathfrak{sl}_2(k)$ in the Lie algebra of derivations whose Lie product defines a $\Z$-grading. 

More precisely, let $B$ be a commutative $k$-domain and $D,U\in {\rm LND}(B)$. 
As noted in the previous section, ${\rm LND}(B)\subset {\rm Der}_k(B)$ and $[D,U]\in {\rm Der}_k(B)$. 
Therefore, for each $d\in\Z$, we have 
\[
[D,U]-dI\in {\rm End}_k(B)
\]
where $I\in {\rm End}_k(B)$ is the identity operator. 
\begin{definition}\label{fun-pair} {\rm Let $B$ be a commutative $k$-domain.
A pair $(D,U)$ of locally nilpotent derivations of $B$ is a {\bf fundamental pair} if 
\begin{enumerate}
\item $[D,[D,U]]=-2D$ and $[U,[D,U]]=2U$, and
\smallskip
\item $B=\sum_{d\in\Z}B_d$ where $B_d=\krn ([D,U]-dI)$.
\end{enumerate}
\smallskip
 Note that (2) is a special form of semisimplicity for $[D,U]$. A fundamental pair $(D,U)$ is {\bf trivial} if $D=U=0$ and {\bf nontrivial} otherwise. 
We say that $D\in {\rm Der}_k(B)$ is a {\bf fundamental derivation} if $D\in {\rm LND}(B)$ and there exists $U\in {\rm LND}(B)$ such that $(D,U)$ is a fundamental pair. 
A $\G_a$-action on $B$ (repectively, ${\rm Spec}(B)$) is {\bf fundamental} if it is induced by a fundamental derivation of $B$.}
\end{definition} 
Observe that the groups $B^*\rtimes\Z_2$ and ${\rm Aut}_k(B)$ act on the set of fundamental pairs for $B$ by
\[
\mu (D,U)=(U,D) \,\, ,\,\, r(D,U)=(rD,r^{-1}U)  \quad {\rm and}\quad \alpha (D,U)=(\alpha D\alpha^{-1},\alpha U\alpha^{-1}) 
\]
where $\Z_2=\langle\mu\rangle$, $r\in B^*$ and $\alpha\in {\rm Aut}_k(B)$, and where $\Z_2$ acts on $B^*$ by $\mu (r)=r^{-1}$. Two fundamental pairs on $B$ are {\bf equivalent} if and only if they lie in the same orbit of this action. 

We first confirm that, for any fundamental pair $(D,U)$ on $B$, the vector spaces $B_d$ define a $\Z$-grading of $B$ and that $D$ and $U$ are homogeneous. 
\begin{lemma}\label{begin} Let $B$ be a commutative $k$-domain with fundamental pair $(D,U)$. Given $d\in\Z$, let $B_d$ denote the kernel of $[D,U]-dI$ as a linear operator on $B$.
\begin{itemize}
\item [{\bf (a)}] $B=\bigoplus_{d\in\Z}B_d$ is a $\Z$-grading.
\smallskip
\item [{\bf (b)}] $D$ and $U$ are homogeneous, $\deg D=2$ and $\deg U=-2$. 
\end{itemize}
\end{lemma}

\begin{proof} Let $E=[D,U]$. 
Given $d,e\in\Z$ and $f\in B_d$, $g\in B_e$ we have
\[
E(fg)=fEg+gEf=f(eg)+g(df)=(d+e)fg\implies fg\in B_{d+e}
\]
and
\[
-2Df = [D,E]f=D(df)-E(Df) \implies E(Df)=dDf+2Df=(d+2)Df \implies Df\in B_{d+2}
\]
and:
\[
2Uf = [U,E]f=U(df)-E(Uf) \implies E(Uf)=U(df)-2Uf=(d-2)f \implies Uf\in B_{d-2}
\]
Therefore, $B_dB_e\subset B_{d+e}$, $DB_d\subset B_{d+2}$ and $UB_d\subset B_{d-2}$. 

Given distinct $d_1,\hdots ,d_n\in\Z$ and nonzero $f_{d_i}\in B_{d_i}$, $1\le i\le n$, we show by induction on $n$ that $f_{d_1},\hdots ,f_{d_n}$ are linearly independent over $k$. 
The case $n=1$ is clear. Assume that $n\ge 2$ and $c_1f_{d_1}+\cdots +c_nf_{d_n}=0$ for $c_1,\hdots ,c_n\in k^*$. Repeated application of $E$ shows:
\[
0=\det\begin{pmatrix} c_1&\cdots &c_n\cr d_1c_1&\cdots &c_nd_n\cr d_1^2c_1&\cdots &d_n^2c_n\cr \vdots &\vdots &\vdots\cr d_1^{n-1}c_1&\cdots &d_n^{n-1}c_n\end{pmatrix}
=c_1\cdots c_n\det\begin{pmatrix} 1&\cdots &1\cr d_1&\cdots &d_n\cr d_1^2&\cdots &d_n^2\cr \vdots &\vdots &\vdots\cr d_1^{n-1}&\cdots &d_n^{n-1}\end{pmatrix}
\]
Since the $d_i$ are distinct, the Vandermonde determinant on the right is nonzero. Therefore, $c_m=0$ for some $m$ and:
\[
c_1f_{d_1}+\cdots +\widehat{c_mf_{d_m}}+\cdots +c_nd_n=0
\]
By the inductive hypothesis we conclude that $c_i=0$ for each $1\le i\le n$. So $f_{d_1},\hdots ,f_{d_n}$ are linearly independent.

By hypothesis, the union of all $B_d$ spans $B$. We have thus shown: $B=\bigoplus_{d\in\Z}B_d$ and this is a $\Z$-grading.
\end{proof}

The grading in this lemma is called the {\bf $\Z$-grading of $B$ induced by $(D,U)$}. 

By this lemma, condition (2) in {\it Definition\,\ref{fun-pair}} can be replaced by the condition:
\begin{quotation}
$(2)^{\prime}$ \,\, $B=\bigoplus_{d\in\Z}B_d$ is a $\Z$-grading where $B_d=\krn ([D,U]-dI)$. 
\end{quotation} 

\begin{example}\label{basic-pair} {\rm 
Let $B=k[x_0,\hdots ,x_n]\cong k^{[n+1]}$ and let $D$ and $U$ be the up and down operators on $kx_0\oplus\cdots\oplus kx_n$ as defined in (\ref{down-op}) above. 
Then $D$ and $U$ extend to locally nilpotent derivations of $B$, and $(D,U)$ is a fundamental pair for $B$, called the {\bf basic} fundamental pair. 
The induced $\Z$-grading of $B$ is defined by declaring that $x_i$ is homogeneous and $\deg x_i=n-2i$, $0\le i\le n$. 
}
\end{example} 

\subsection{The Structure Theorem}
\begin{theorem}\label{main1} {\bf (Structure Theorem for Fundamental Derivations)} Let $B$ be an affine $k$-domain, let $(D,U)$ be a nontrivial fundamental pair for $B$, and let 
$B=\bigoplus_{i\in\Z}B_i$ be the induced $\Z$-grading of $B$ with degree function $\deg$. 
Let $A=\krn D$ and $\Omega=\krn U$ and, for each $n\in\Z$, define $A_n=A\cap B_n$ and $\Omega_n=\Omega\cap B_n$. For each $n\ge 0$ define
$\mathcal{F}_n=\krn D^{n+1}$ and $I_n=D^n(\mathcal{F}_n)$.
\begin{itemize}
\item [{\bf (a)}] $A$ is affine. 
\smallskip
\item [{\bf (b)}]  $I_n=\bigoplus_{i\ge n}A_i$ for each $n\ge 0$. In particular, $A=\bigoplus_{i\in\N}A_i$ is $\N$-graded. 
\smallskip
\item [{\bf (c)}] Let $f_1,\hdots ,f_r\in A$ be homogeneous such that $I_1=(f_1,\hdots ,f_r)$. Set $m=\max_i\deg f_i$. Then:
\smallskip
\subitem {\bf (i)} $A=A_0[f_1,\hdots ,f_r]$  
\smallskip
\subitem {\bf (ii)} Given $n\ge 1$, $I_n=(A_n,\hdots ,A_{n+m-1})$. 
\smallskip
\subitem {\bf (iii)} Given $n>m$, $\displaystyle I_n=\prod_E I_1^{e_1}\cdots I_m^{e_m}$ where $E=\{ (e_1,\hdots ,e_m)\in\N^m\, |\, \sum_iie_i=n\}$.
\medskip
\item [{\bf (d)}] Given $n\ge 1$, assume that $I_n=(g_1,\hdots ,g_s)$ for homogeneous $g_i\in A$ and set $h_i=U^ng_i$, $1\le i\le s$. Then:
\[
\mathcal{F}_n=Ah_1+\cdots +Ah_s+\mathcal{F}_{n-1} 
\]
\item [{\bf (e)}] $B=\Omega\oplus DB=A\oplus UB$ as $A_0$-modules.
\smallskip
\item [{\bf (f)}] $D:B_n\to B_{n+2}$ is injective if $n\le -1$ and surjective if $n\ge -1$.
\smallskip
\item [{\bf (g)}] $B_0=A_0\oplus M$ as $A_0$-modules, where $M=DB_{-2}=UB_2$.
\end{itemize}
\end{theorem}

Three preliminary lemmas are needed to prove the theorem. We continue the notation used in the hypotheses of the theorem. 

\begin{lemma}\label{updown} {\rm (\cite{Freudenburg.13}, Lemma 3.2)} 
Given homogeneous $f\in A$ define $c_n\in\Z$ by $c_n=n\deg f-n(n-1)$, $n\ge 1$. 
\begin{itemize}
\item [{\bf (a)}] $D^mU^nf=c_nD^{m-1}U^{n-1}f$ for all $m,n\ge 1$.
\smallskip
\item [{\bf (b)}] $D^nU^nf=c_1\cdots c_nf$ for all $n\ge 1$.
\end{itemize}
\end{lemma}
\begin{proof} Set $E=[D,U]$. 
The $k$-derivation $\delta:={\rm ad}(D)$ on $\mathfrak{sl}_2(k)$ is locally nilpotent and acts on the algebra generators by:
\[
\delta : U\to E\to (-2D)\to 0
\]
The relation $[U,E]=2U$ easily generalizes to $[U^n,E]=2nU^n$ for all $n\ge 1$. We claim that
\begin{equation}\label{eq1}
\textstyle\delta (U^n)=nU^{n-1}(E-(n-1)I) \quad {\rm for \,\, all} \,\, n\ge 1.
\end{equation}
This is clear if $n=1$, so assume it holds for $n\ge 1$. Then
\begin{eqnarray*}
\delta (U^{n+1})&=&U\delta (U^n)+\delta (U)U^n \\
&=& \textstyle nU^n(E-(n-1)I)+EU^n \\
&=& \textstyle nU^n(E-(n-1)I)+(U^nE-2nU^n) \\
&=& \textstyle (n+1)U^n(E-nI)
\end{eqnarray*}
So equation (\ref{eq1}) is confirmed by induction on $n$. 
In addition, for all $m,n\ge 1$:
\[
D^mU^n-D^{m-1}U^nD=D^{m-1}(DU^n-U^nD)=D^{m-1}\delta (U^n)
\]
By (\ref{eq1}) it follows that, for all $m,n\ge 1$:
\begin{equation}\label{eq2}
\textstyle D^mU^n=D^{m-1}U^{n-1}\left( UD+nE-n(n-1)I\right)
\end{equation}
Therefore, if $f\in A$ is homogeneous then (\ref{eq2}) implies:
\[
D^mU^nf=c_nD^{m-1}U^{n-1}f
\]
This proves part (a) and part (b) follows inductively from part (a) using $m=n$. 
\end{proof}

\begin{lemma}\label{degrees} Let $f\in A$ be homogeneous. 
\begin{itemize}
\item [{\bf (a)}] $\deg_Uf=\deg f$. Consequently, the $\Z$-grading of $B$ restricts to an $\N$-grading of $A$.  
\smallskip
\item [{\bf (b)}] $\mathcal{F}_n\cap B_{-n}=\Omega_{-n}$ for each $n\ge 1$. 
\item [{\bf (c)}] $U\mathcal{F}_n\subset\mathcal{F}_{n+1}$ for each $n\ge 0$.
\smallskip
\item [{\bf (d)}] The ideal $fB+UfB+\cdots +U^dfB$ is $(D,U)$-invariant, where $d=\deg f$.
\end{itemize}
\end{lemma}

\begin{proof} We may assume that $f\ne 0$. If $N=\deg_Uf$ then $U^Nf\ne 0$ and $U^{N+1}f=0$. By {\it Lemma\,\ref{updown}(a)} we have:
\[
0=DU^{N+1}f=c_{N+1}U^Nf=(N+1)(\deg f-N)U^Nf \implies \deg f=N
\]
Therefore, $\deg f=\deg_Uf\in\N$. This proves part (a). 

Given $n\ge 1$, any nonzero $h\in\Omega_{-n}$ has $\deg_Dh=n$ by part (a) and symmetry of $D$ with $U$. So $\Omega_{-n}\subset\mathcal{F}_n\cap B_{-n}$. 

For the reverse inclusion, we show by induction on $i$ that $\mathcal{F}_i\cap B_{-n}\subset\Omega$ for $0\le i\le n$. Since $A_{-n}=\{ 0\}$ by part (a), it follows that:
\[
\mathcal{F}_0\cap B_{-n}=A\cap B_{-n}=A_{-n}=\{0\}
\]
This gives the basis for induction. 

Assume that $\mathcal{F}_{i-1}\cap B_{-n}\subset\Omega$ for some $1\le i\le n$. 
Part (a) shows:
\[
h\in \mathcal{F}_i\cap B_{-n}\setminus\{ 0\} \implies D^ih\in A_n\setminus\{ 0\} \implies U^iD^ih\in\Omega_{-n}\setminus\{ 0\}
\]
In addition, {\it Lemma\,\ref{updown}} shows that there exists $g\in\Omega_{-n}$ such that $U^iD^ih=U^iD^ig$. Therefore, $U^iD^i(h-g)=0$. Since $U^ip\ne 0$ for any nonzero $p\in A_n$ (by part (a)), we must have $D^i(h-g)=0$, i.e., 
$h-g\in\mathcal{F}_{i-1}\cap B_{-n}$. By the inductive hypothesis, $h-g\in\Omega$ so $h\in\Omega$. If follows by induction that $\mathcal{F}_n\cap B_{-n}\subset\Omega_{-n}$.
This proves part (b). 

Since $f\in A=\mathcal{F}_0$ we have $DUf=(\deg f)f$ so $Uf\in\mathcal{F}_1$. Assume that $U\mathcal{F}_{n-1}\subset\mathcal{F}_n$ for some $n\ge 1$. 
Given homogeneous $g\in\mathcal{F}_n$, $Dg\in\mathcal{F}_{n-1}$ implies $UDg\in\mathcal{F}_n$. It follows that:
\[
D^{n+2}Ug=D^{n+1}(UD+E)g=D^{n+1}(UDg)+(\deg g)D^{n+1}g=0 \implies  Ug\in\krn D^{n+2}=\mathcal{F}_{n+1}
\]
Part (c) now follows by induction on $n$. 

If $f\in A_d$ for $d\in\N$, then part (a) and {\it Lemma\,\ref{begin}(b)} imply $U^df\in\Omega_{-d}$. 
In addition, {\it Lemma\,\ref{updown}(a)} shows that $DU^if=i(d-i+1)U^{i-1}f$ for each $i$, $0\le i\le d$. 
Therefore, if $J=fB+UfB+\cdots +U^dfB$, then $DJ\subset J$, and clearly $UJ\subset J$ as well. This proves part (d). 
\end{proof}

\begin{lemma}\label{classical} The following statements hold. 
\begin{itemize}
\item [{\bf (a)}] $A_0=A\cap\Omega =\Omega_0$
\item [{\bf (b)}] $A_0$ is factorially closed in $B$ and $ML(B)\subseteq A_0$. 
\item [{\bf (c)}] $\Omega\cap DB=\{ 0\}$
\item [{\bf (d)}] $B^{SL_2(k)}=A_0$ for the $SL_2(k)$-action on $B$ induced by $(D,U)$. 
\end{itemize}
\end{lemma}

\begin{proof} By {\it Lemma\,\ref{degrees}(a)} we have
\[
f\in A_0\setminus\{ 0\} \implies \deg_Uf=0 \implies f\in\Omega
\]
and for $d>0$:
\[
f\in A_d\setminus\{ 0\} \implies \deg_Uf>0 \implies f\not\in\Omega
\]
Therefore, $A_0=A\cap\Omega$ and by symmetry $\Omega_0=A\cap\Omega$. This proves part (a). Part (b) follows from part (a) since the intersection of factorially closed subrings is factorially closed, and any intersection of kernels of locally nilpotent derivations contains $ML(B)$ by definition. 

If $g\in A\cap UB$ and $g\ne 0$ then $Dg=0$ and $g=Uf$ for some nonzero $f\in B$. Therefore
\[
0=D(Uf)=Ef=(\deg f)f \implies \deg f=0 \implies f\in A_0=\Omega_0 \implies g=Uf=0
\]
which gives a contradiction. Therefore, no such $g$ exists and $A\cap UB=\{ 0\}$. Part (c) follows by symmetry. 

For part (d), define subgroups $G_1,G_2\subset SL_2(k)$ by:
\[
G_1=\{ \exp (tD)\, |\, t\in k\}\cong\G_a  \quad {\rm and}\quad G_2=\{ \exp (sU)\, |\, s\in k\}\cong\G_a 
\]
Combining part (a) with our previous observations gives $B^{SL_2(k)}=B^{G_1}\cap B^{G_2}= A\cap\Omega = A_0$.
\end{proof}

We can now give the proof of {\it Theorem\,\ref{main1}}. 
\begin{proof} For part (a), the fact that $A$ is affine is proved in {\it Proposition\,\ref{isomorphism}} below. 

For part (b), by {\it Lemma\,\ref{degrees}(a)}, the $\Z$-grading of $B$ restricts to an $\N$-grading on $A$.
Define $A$-ideals 
\[
J_n=\bigoplus_{i\ge n}A_i \,\, ,\,\, n\ge 0
\]
noting that $J_0=A=I_0$. Given $e\ge 1$ and nonzero $f\in A_e$, if $c_i\in k$ are defined by $c_i=i(e-i+1)$,
then $c_i\ne 0$ for $1\le i\le e$. By {\it Lemma\,\ref{updown}(b)}, $D^e(U^ef)=c_1\cdots c_ef$ so $f\in I_e$ and $A_e\subset I_e$. 
Therefore, $J_n\subseteq I_n$ for all $n\ge 0$. 

We proceed to show by induction on $n$ that $I_n\subseteq J_n$ for each $n\ge 0$. 

A basis for induction is established by the equalities $I_0=A=J_0$.

Assume that, for some $n\ge 0$, $I_i=J_i$ for all $0\le i\le n$. Let $f\in I_{n+1}\cap A_e$ for some $e\ge 0$. If $e\ge n+1$ then $f\in J_{n+1}$. 

Consider the case $e\le n$. If $e=0$ then by {\it Lemma\,\ref{classical}(c)} we have:
\[
f\in I_{n+1}\cap A_0\subset DB\cap\Omega =\{0\}
\]
So we may assume $e\ge 1$. 
Since $f\in I_{n+1}$ there exists $g\in\mathcal{F}_{n+1}$ with $D^{n+1}g=f$. In addition, by {\it Lemma\,\ref{updown}(b)}, we have
$D^eU^ef=cf$ where $c=c_1\cdots c_e\ne 0$. Let $h=D^{n-e+1}(cg)-U^ef$. Then $D^eh=0$. 
Since $\deg D=2$ by {\it Lemma\,\ref{begin}(b)}, we see that $h\in B_{-e}$. 
The inductive hypothesis implies that:
\[
D^{e-1}h\in A_{e-2}\cap I_{e-1}=\{ 0\}
\]
Repeating this argument, we obtain $D^{e-i}h\in A_{e-2i}\cap I_{e-i}=\{ 0\}$ for $1\le i\le e$, so $h\in A_{-e}=\{ 0\}$. By {\it Lemma\,\ref{degrees}(a)}, $\deg_Uf=e$ which implies
$U^ef\in\Omega$ and $U^ef\ne 0$. Since $n-e+1>0$, we see from {\it Lemma\,\ref{classical}(c)} that
\[
U^ef=D^{n-e+1}(cg)\in\Omega\cap DB=\{ 0\}
\]
which gives a contradiction. So this case cannot occur. 

By induction we obtain $I_n\subseteq J_n$ for all $n\ge 0$. Consequently, $J_n=I_n$ for all $n\ge 0$, thus proving part (b). 

For part (i) of (c), it is well-known that, since $A$ and $A_0$ are affine, $A=A_0[f_1,\hdots ,f_r]$ for any set $\{ f_1,\hdots ,f_r\}$ of generators of the irrelevant ideal 
$\bigoplus_{n\ge 1}A_n$, which is the ideal $I_1$ bv part (a). For part (ii) of (c), given $n\ge 1$, let $J$ be the ideal generated by $A_n,\hdots ,A_{n+m-1}$. Then $J\subset I_n$. For the reverse inclusion, we show by induction on $i$ that $A_{n+i}\subset J$ for $i\ge 0$. This holds by definition for $0\le i\le m-1$. Assume $A_{n+i}\subset J$ for some $i\ge m-1$ and let $f=f_1^{d_1}\cdots f_r^{d_r}$ be a monomial in $A_{n+i+1}$. Choose $j$ so that $d_j\ne 0$ and set $g=f/f_j$. Then $g$ is homogeneous and 
$\deg g=n+i+1-\deg f_j\ge n+i-(m-1)\ge n$, which implies $g\in J$. So $f=f_jg\in J$ and (ii) is proved by induction. Part (iii) of (c) follows from (i) and (ii). 

For part (d), since $g_i\in I_n$, part (b) implies that $\deg g_i\ge n$. By {\it Lemma\,\ref{degrees}} we have:
\[
1\le n\le\deg g_i=\deg_Ug_i
\]
Therefore, $h_i=U^ng_i\ne 0$ for each $i$, 
and by {\it Lemma\,\ref{updown}}, $I_n=(D^nh_1,\hdots ,D^nh_s)$. This shows:
\[
\mathcal{F}_n=Ah_1+\cdots +Ah_s+\mathcal{F}_{n-1}
\]
This proves part (d).

For part (e), {\it Lemma\,\ref{classical}(c)} shows that, given $n\ge 0$, if $M_n\subset\mathcal{F}_n$ is the submodule 
\[
M_n=(\mathcal{F}_n\cap\Omega)+(\mathcal{F}_n\cap DB)
\]
then $M_n=(\mathcal{F}_n\cap\Omega )\oplus (\mathcal{F}_n\cap DB)$. It will thus suffice to show $M_n=\mathcal{F}_n$. Note that $M_{n-1}\subset M_n$ if $n\ge 1$. 

If $n=0$ then part (b) implies $\mathcal{F}_n=A_0\oplus I_1=(\mathcal{F}_0\cap\Omega )\oplus D\mathcal{F}_1$. Assume by way of induction that
$\mathcal{F}_{n-1}=M_{n-1}$ for some $n\ge 1$. 
Choose nonzero homogeneous $g\in\mathcal{F}_n$, say $g\in B_d$ for $d\in\Z$. 
Then $D^ng\in A_{d+2n}\cap I_n$ where $D^ng\ne 0$. Part (b) implies $d+2n\ge n$, so $d+n\ge 0$. 
By {\it Lemma\,\ref{degrees}} we have $\deg_U(D^ng)=d+2n$. 
By {\it Lemma\,\ref{updown}} we have
\begin{equation}\label{bigDU}
D^{d+2n}U^{d+2n}(D^ng)=cD^ng
\end{equation}
for some nonzero $c\in k$. 

If $d+n=0$ then $\deg_U(D^ng)=n$ implies $U^nD^ng\in\Omega$. Therefore (\ref{bigDU}) shows:
\[
U^nD^ng-cg\in\krn D^n=\mathcal{F}_{n-1} \implies g\in (\mathcal{F}_n\cap\Omega )+\mathcal{F}_{n-1}\subset M_n
\]

If $d+n\ge 1$ then (\ref{bigDU}) shows:
\[
D^{d+n}U^{d+2n}(D^ng)-cg\in\krn D^n=\mathcal{F}_{n-1}\implies g\in D\mathcal{F}_{n+1}+\mathcal{F}_{n-1}\subset M_n
\]
Therefore, $M_n=\mathcal{F}_n$ for all $n\ge 0$. It follows that 
$B=\Omega\oplus DB$. The equality $B=A\oplus UB$ follows from this by symmetry of $D$ with $U$. So part (e) is proved. 

The fact that $\deg D=2$ was established in {\it Lemma\,\ref{begin}}. The fact that $D:B_n\to B_{n+2}$ is injective for $n<0$ follows from part (b). 
Let $g\in B_d$ for some $d>0$. By part (d) we can write $g=p+Db$ for homogeneous $p\in\Omega$ and $b\in B$. But then $p\in\Omega_d=\{ 0\}$, so $g=Db$. Therefore, 
$B_d\subset DB$ when $d\ge 1$, so $D:B_n\to B_{n+2}$ is surjective when $n\ge -1$. So part (f) is proved. 

Since $B_0=\krn E$ for $E=DU-UD$ we see that $DUp=UDp$ for each $p\in B_0$. Let $f\in B_2$ be given. By part (f) we have $DB_0=B_2$, so there exists $h\in B_0$ with $f=Dh$. 
Therefore, $Uf=UDh=DUh\in DB_{-2}$, so $UB_2\subset DB_{-2}$. By symmetry, $DB_{-2}\subset UB_2$, so
$DB_{-2}=UB_2$. Part (g) now follows from part (e). 
\end{proof}

\begin{corollary}\label{plinth-cor} Let $B$ be an affine $k$-domain with $\dim_k(B)\ge 2$. Under the hypotheses of {\it Theorem\,\ref{main1}}:
\begin{itemize}
\item [{\bf (a)}] $A=A_0\oplus I_1$, $A_0$ is a retract of $A$ and $I_1$ is a prime ideal of $A$.
\smallskip
\item [{\bf (b)}] For each $a\in A_0$ and integer $n\ge 0$, $aA\cap I_n=aI_n$.
\smallskip
\item [{\bf (c)}] Given $d\in\N$ and $f,g\in A_d$, $[f,g]_d^U\in A_0$.
\end{itemize}
\end{corollary}

\begin{proof} Part (a) follows directly from {\it Theorem\,\ref{main1}(b)}.

Given $a\in A_0$, if $a=0$ then $aA\cap I_n=aI_n$ holds, so we may assume $a\ne 0$. Suppose that $f=ag$ for $f\in I_n$ and $g\in A$. 
{\it Theorem\,\ref{main1}(b)} implies $\deg f\ge n$. 
Since $\deg a=0$ it follows that $\deg g\ge n$. By {\it Theorem\,\ref{main1}(b)} we see that $g\in I_n$, thus proving part (b). 

For part (c), we have $\deg_Uf=\deg_Ug=d$ by {\it Lemma\,\ref{degrees}}. Therefore, $[f,g]_d^U\in\Omega$ and:
\[
\deg [f,g]_d^U=\deg fU^dg=d+(d-2d)=0
\]
So $[f,g]_d^U\in\Omega_0=A_0$.
\end{proof}

\begin{corollary}\label{no-fun} Let $B$ be an affine $k$-domain, $\delta\in {\rm LND}(B)$, $\delta\ne 0$, $A=\krn\delta$ and $I=A\cap \delta B$. Suppose that at least one of the following conditions holds.
\begin{enumerate}
\item $I$ is not a prime ideal of $A$.
\item $I$ is a prime ideal of $A$ but there is no retraction of $A$ with kernel $I$.
\item $\delta =f\delta^{\prime}$ where $\delta^{\prime}\in {\rm LND}(B)$ does not have a slice and $f\in A\setminus A^*$.
\end{enumerate}
Then $\delta$ is not fundamental. 
\end{corollary}

\begin{proof} If $I$ is not a prime ideal of $A$ then {\it Corollary\,\ref{plinth-cor}(a)} implies that $\delta$ is not fundamental. 
If there is no retraction of $A$ with kernel $I$ then {\it Corollary\,\ref{plinth-cor}(b)} implies that $\delta$ is not fundamental. 

Assume that the hypotheses (3) hold. Then $I=fI^{\prime}$ where $I^{\prime}=A\cap \delta^{\prime}B$. Since $fI^{\prime}\subset I^{\prime}$ we see that $af\in I$ for every $a\in I^{\prime}$. 

Suppose that $I$ is a prime ideal of $A$. If $f\not\in I$ then $a\in I$ for every $a\in I'$, which implies that $fI^{\prime}=I^{\prime}$. Since $I^{\prime}\ne (0)$ and $B$ is affine it follows that $f\in A^*$, a contradiction. Therefore, $f\in I$, which implies $1\in I^{\prime}$, also a contradiction. 
So $I$ is not a prime ideal of $A$ and {\it Corollary\,\ref{plinth-cor}(a)} implies that $\delta$ is not fundamental. 
\end{proof} 

\begin{example} {\rm Every linear $\G_a$-action on $\A_k^n$ is fundamental but this is not generally true for quasi-linear actions. 
Recall that a $k$-derivation $\delta$ of $k[x_1,\hdots ,x_n]=k^{[n]}$ is {\bf quasi-linear} if there is an $n \times n$ matrix $M$ with coefficients in $\krn\delta$ such that 
$[\delta x_1 \cdots \delta x_n]=[x_1 \cdots x_n]M$. A $\G_a$-action on $\A_k^n$ is quasi-linear if it is induced by a quasi-linear locally nilpotent derivation. 

Let $B=k[x_0,x_1,y_0,y_1]\cong k^{[4]}$ with linear fundamental  pair
\[
 (D,U)=(x_0\partial_{x_1}+y_0\partial_{y_1}, x_1\partial_{x_0}+y_1\partial_{y_0})
 \]
and set $A=\krn D$ and $A_0=A\cap\krn U$. Then $A=k[x_0,y_0,P]$ and $A_0=k[P]$ for $P=x_0y_1-y_0x_1$. The derivation $\delta :=PD$ is quasi-linear, since:
\[
\delta x_0=\delta y_0 =0 \quad {\rm and}\quad 
\delta \begin{pmatrix} x_1\cr y_1\end{pmatrix} = \begin{pmatrix} -x_0y_0 & x_0^2\cr -y_0^2 & x_0y_0 \end{pmatrix} \begin{pmatrix} x_1\cr y_1\end{pmatrix}
\]
Since $A\cap\delta B=(x_0P,y_0P)$ is not a prime ideal of $A$, {\it Corollary\,\ref{no-fun}} implies that $\delta$ is not fundamental.}
\end{example}

\begin{lemma}\label{generators} Let $B$ be an affine $k$-domain with nontrivial fundamental pair $(D,U)$. Suppose $A=A_0[x]$ for some $x\in A_d$, $d\ge 1$, noting that
$A_n=A_0\cdot x^i$ if $n=di$, $i\ge 0$, and $A_n=\{ 0\}$ for $n\not\in d\Z$. 
\begin{itemize}
\item [{\bf (a)}] $D$ has the Freeness Property. In particular:
\[
I_n=\bigoplus_{i\ge n}A_i=\begin{cases} \left (x^{n/d}\right) & n\in d\Z \\ \left( x^{[n/d]+1}\right) & n\not\in d\Z \end{cases}
\]
\item [{\bf (b)}] $\mathcal{F}_n= \bigoplus_{-1\le i\le n}A\cdot h_i$ where $h_{-1}=1$, $h_0=x$ and:
\[
h_i=\begin{cases} U^i\left( x^{i/d}\right) & i\in d\Z \\ U^i\left( x^{[i/d]+1}\right) & i\not\in d\Z \end{cases} \quad (i\ge 1)
\]
\item [{\bf (c)}] $B=A_0[h_0,\hdots ,h_d]=R\oplus Rh_1\oplus \cdots\oplus Rh_{d-1}$ where $R=A_0[h_0,h_d]\cong A_0^{[2]}$.
\end{itemize}
\end{lemma}

\begin{proof} Part (a) is implied by the Structure Theorem, and part (b) is an immediate consequence of part (a). 

For part (c), since $D$ has the Freeness Property, any set $\{ b_n\}_{n\in\N}$ in $B$ such that $(D^nb_n)=I_n$ for each $n\ge 0$ is an $A$-module basis of $B$. In particular, if $J=\N\times\{ 0,\hdots ,d-1\}$, then:
\[
B=\bigoplus_{(n,i)\in J}A\cdot h_d^nh_i=R\oplus Rh_1\oplus \cdots\oplus Rh_{d-1}
\]
\end{proof}

\begin{lemma}\label{Ksl2} Let $B$ be an affine $k$-domain with fundamental pair $(D,U)$. 
Let $\mathcal{F}_n=\krn D^{n+1}$, $n\ge 0$, and let $N\ge 1$ be such that $B=k[\mathcal{F}_N]$. 
Given nonzero $a\in A_0$ define submodules 
\[
\mathcal{G}_n=\sum_{0\le i\le n}a^i\mathcal{F}_i \,\, ,\,\, n\ge 0
\]
and define the subring $R=k[\mathcal{G}_N]$. 
\begin{itemize}
\item [{\bf (a)}] $R$ is a graded subring of $B$, $R\cap\mathcal{F}_n=\mathcal{G}_n$ for each $n\ge 0$, and $(a^{-1}D,aU)$ is a fundamental pair for $R$. 
\item [{\bf (b)}] If $B$ is normal then $R$ is normal.
\item [{\bf (c)}] If $B$ is a UFD then $R$ is a UFD.
\end{itemize}
\end{lemma}

\begin{proof} The Structure Theorem shows that 
each module $\mathcal{F}_n$ is generated by homogeneous elements. Since $a$ is also homogeneous, we see that $R$ is a graded subring of $B$. 

Consider a monomial $\mu\in\mathcal{F}_n$ for some $n\ge N+1$. Since $B=k[\mathcal{F}_N]$ there exist $x_i\in\mathcal{F}_{n_i}$ with $n_i\le N$ and $e_i\in\N$, $1\le i\le m$, such that:
\[
\mu =x_1^{e_1}\cdots x_m^{e_m} \quad {\rm and}\quad \sum_{1\le i\le m} e_in_i=n
\]
Therefore, $a^n\mu = (a ^{n_1}x_1)^{e_1}\cdots (a^{n_m}x_m)^{e_m}\in R$. 
It follows that $a^n\mathcal{F}_n\subset R$ for each $n\ge 0$, which implies $\mathcal{G}_n\subset R$ for each $n\ge 0$. 

Since $a\in A_0$, $(D,U)$ extends to a fundamental pair on $B_a$. Since $a$ is a unit of $B_a$, $(a^{-1}D,aU)$ is a fundamental pair for $B_a$. 
In addition, given $n\ge 0$, {\it Lemma\,\ref{degrees}(c)} implies
\[
a^{-1}D(a^i\mathcal{F}_i )= a^{i-1}D\mathcal{F}_i\subset a^{i-1}\mathcal{F}_{i-1} \,\, (i\ge 1) \quad {\rm and}\quad a^{-1}D(\mathcal{F}_0)=0
\]
and:
\[
aU(a^i\mathcal{F}_i)=a^{i+1}U\mathcal{F}_i\subset a^{i+1}\mathcal{F}_{i+1} \,\, (i\ge 0) 
\]
Therefore, $a^{-1}D$ and $aU$ restrict to $R$, and $(a^{-1}D,aU)$ is a fundamental pair for $R$. This proves part (a).

Parts (b) and (c) are consequences of {\it Proposition\,\ref{main2}} below. 
\end{proof} 

%%%%%%%%%%%%%%%%%%%%%%%%%%%%%%%%%%%%%%%%%%%%%%%%%%%%%%%%%%%%%%%%%%%%%%%%%%%%%%

\section{The Classification Theorem}\label{surfaces} This section classifies normal affine $k$-domains of dimension two with trivial units admitting a nontrivial fundamental pair, i.e., a nontrivial $SL_2(k)$-action.
Our method requires first a detailed understanding of certain determinantal ideals which are invariant for basic fundamental pairs. 

\subsection{Quadratic Determinantal Ideals}\label{basic}
Let $B=k[x_0,\hdots ,x_d]\cong k^{[d+1]}$ for $d\ge 2$, let $(D,U)$ be the basic fundamental pair on $D$, let $A=\krn D$ and let 
$B=\bigoplus_{n\in\Z}B_n$ be the induced grading of $B$. In particular, each $x_i$ is homogeneous of degree $d-2i$. 
Let $J\subset B$ be the ideal generated by the $2\times 2$ quadratic minors of the $2\times d$ matrix
\[
\begin{pmatrix} x_0&x_1&x_2&\cdots&x_{d-1}\cr x_1&2x_2&3x_3&\cdots&dx_d\end{pmatrix}
\]
noting that $J$ is a graded ideal. 
Let $M_{a,b}$ be the minor using columns $a$ and $b$, $1\le a<b\le d$, that is:
\[
M_{a,b}=\det\begin{pmatrix} x_{a-1} & x_{b-1} \cr ax_a & bx_b \end{pmatrix}=bx_{a-1}x_b-ax_ax_{b-1}
\]
Let $\Theta ={\rm span}_k\{ M_{a,b}\, |\, 1\le a<b\le d\}$ and let $\Theta_n=\Theta\cap B_n$, $n\in\Z$.

It is well known that $J$ is a prime ideal defining the affine cone $X_d\subset\A_k^{d+1}$ over the rational normal curve of degree $d$; see 
\cite{Eisenbud.05}, Proposition 6.1. 
The corresponding Eagon-Northcott complex is a minimal free resolution of $B/J$ of length $d$ in which each kernel is generated by linear forms \cite{Eagon.Northcott.62}. 
In particular, the module of first syzygies for $B/J$ is generated by the $3\times 3$ minors of the matrices:
\[
\begin{pmatrix} x_0&x_1&x_2&\cdots&x_{d-1}\cr x_1&2x_2&3x_3&\cdots&dx_d \cr  x_0&x_1&x_2&\cdots&x_{d-1}\end{pmatrix}
\quad {\rm and}\quad 
\begin{pmatrix} x_0&x_1&x_2&\cdots&x_{d-1}\cr x_1&2x_2&3x_3&\cdots&dx_d \cr x_1&2x_2&3x_3&\cdots&dx_d\end{pmatrix}
\]
The surface $X_d$ admits a natural $SL_2(k)$-action, an action which is investigated in the next section. 
Thus, $J$ (and its associated complex) are $(D,U)$-invariant. We give a $(D,U)$-invariant basis of $J$.

Let $m\in\N$ be such that $d=2m$ or $d=2m+1$. Define homogeneous $T_2^{(0)},T_4^{(0)},\hdots ,T_{2m}^{(0)}\in A$ by: 
\begin{equation}\label{quadratic}
T_{2i}^{(0)}=\sum_{0\le j\le 2i}(-1)^jx_jx_{2i-j}  \quad {\rm where}\quad \deg T_{2i}^{(0)}=2d-4i
\end{equation}
Let $\widehat{T}_{2i}$ denote the $D$-cable rooted at $T_{2i}^{(0)}$, which is defined by generators:
\[
T_{2i}^{(j)}=U^jT_{2i}^{(0)} \,\, , \,\, 0\le j\le 2d-4i
\]
Note that this cable is $(D,U)$-invariant, since 
\[
UT_{2i}^{(j)}=T_{2i}^{(j+1)} \,\, , \,\, 0\le j\le 2d-4i-1
\]
and:
\[
DT_{2i}^{(j)}=t_iT_{2i}^{(j-1)} \,\, {\rm for}\,\, t_i\in k^* \,\, , \,\, 1\le j\le 2d-4i
\]
Let $\{ \widehat{T}_2,\widehat{T}_4,\hdots ,\widehat{T}_{2m}\}$ be the set of vertices $T_{2i}^{(j)}$ of these cables and let 
$(\widehat{T}_2,\widehat{T}_4,\hdots ,\widehat{T}_{2m})$ be the ideal of $B$ generated by this set.

\begin{proposition}\label{quadratic-ideal} Assume that $d\ge 2$.
\begin{itemize}
\item [{\bf (a)}]  $J=(\widehat{T}_2,\widehat{T}_4,\hdots ,\widehat{T}_{2m})$
\item [{\bf (b)}] The sets $\{ M_{a,b}\, |\, 1\le a<b\le d\}$ and $\{ \widehat{T}_2,\widehat{T}_4,\hdots ,\widehat{T}_{2m}\}$ are bases for $\Theta$.
\end{itemize}
\end{proposition}

\begin{proof} We first show that $(\widehat{T}_2,\widehat{T}_4,\hdots ,\widehat{T}_{2m})\subseteq J$. 

Let $F_d\subset k[y,z]\cong k^{[2]}$ be the vector space of $d$-forms, with basis 
\[
\{ f_i=y^iz^{d-i}\, |\, 0\le i\le d\}
\]
and let $\varphi :B\to k[F_d]$ map $x_i\to f_i$. Then $k[F_d]$ is the coordinate ring of $X_d$ and $\krn\varphi =J$. 
A fundamental pair $(D',U')$ on $k[F_d]$ is given by restriction of the the basic fundamental pair $(y\partial_z,z\partial_y)$ on $k[y,z]$. 
In particular, $\krn D'=k[y]\cap k[F_d]=k[f_0]\cong k^{[1]}$ and for this pair, $\varphi$ is equivariant. 

Since $T_{2i}^{(0)}\in A$ for each $i$, we see that:
\[
\varphi (T_{2i}^{(0)})\in\krn D'=k[f_0]\,\, ,\quad 1\le i\le m
\]
By homogeneity, there exist $r_i\in k$ and $e_i\in\N$ such that $\varphi (T_{2i}^{(0)})=r_if_0^{e_i}$, $1\le i\le m$. If $r_i\ne 0$ for some $i$, then degree considerations show that either
\begin{enumerate}
\item $d=2m$, $i=m$ and $\varphi (T_d^{(0)})=r\in k^*$, or
\item $d=2m$, $2i=m$ and $\varphi (T_m^{(0)})=rf_0$ for $r\in k^*$.
\end{enumerate}
But then either $1\in (f_0,\hdots ,f_d)$ or $f_0\in (f_0,\hdots ,f_d)^2$ in $k[y,z]$, a contradiction. Therefore, $r_i=0$ for each $i$ and 
$(T_2^{(0)},T_4^{(0)},\hdots ,T_{2m}^{(0)})\subset J$. 
Since $J$ is $(D,U)$-invariant it follows that $(\widehat{T}_2,\widehat{T}_4,\hdots ,\widehat{T}_{2m})\subseteq J$.

Let $W$ be the $k$-span of the vertices of $\widehat{T}_2,\widehat{T}_4,\hdots ,\widehat{T}_{2m}$. 
Since $(\widehat{T}_2,\widehat{T}_4,\hdots ,\widehat{T}_{2m})\subseteq J$ it must be the case that $W\subseteq\Theta$ (by homogeneity). 
According to \cite{Freudenburg.Kuroda.17}, Lemma 3.8, the vertices of $\hat{T}_2,\hat{T}_4,\hdots ,\hat{T}_{2m}$ are linearly independent over $k$. The number of such vertices equals:
\[
\sum_{i=1}^m 2d-4i+1 ={d\choose 2}
\]
Since the number of minors $M_{a,b}$ also equals $d\choose 2$ we conclude that the $M_{a,b}$ are linearly independent and $W=\Theta$. Consequently, 
$(\hat{T}_2,\hat{T}_4,\hdots ,\hat{T}_{2m})=J$. This proves parts (a) and (b). 
\end{proof}

\begin{lemma}\label{d=2m} Assume that $d\in2\Z$. 
\begin{itemize}
\item [{\bf (a)}] If $r,s\in\N$ and $r+s=d$ then $x_rx_s\not\in UB$. 
\item [{\bf (b)}] If $a,b\in\N$ is a pair such that $M_{a,b}\in \Theta_0$ then $M_{a,b}\not\in UB$. 
\item [{\bf (c)}] The set $\{T_{2i}^{(d-2i)}\, |\, 1\le i\le \frac{d}{2}\}$ is a basis for $\Theta_0$. 
\smallskip
\item [{\bf (d)}] $\Theta_0=(\Theta_0\cap UB)\oplus k\cdot T_d^{(0)}$ where $\{ T_2^{(d-2)} , \hdots , T_{d-2}^{(2)}\}$ is a basis for $\Theta_0\cap UB$. 
\end{itemize}
\end{lemma}

\begin{proof} 
Consider the set $S$ of all pairs $(r,s)\in\N^2$ on the line $r+s=d$ such that $x_rx_s\in UB$. If $(r,s)\in S$ and $r\ge 1$ then:
\[
U(x_{r-1}x_s)=r(d-r+1)x_rx_s+(s+1)(d-s)x_{r-1}x_{s+1} \implies x_{r-1}x_{s+1}\in UB
\]
Similarly, if $(r,s)\in S$ and $s\ge 1$ then:
\[
U(x_rx_{s-1})=(r+1)(d-r)x_{r+1}x_{s-1}+s(d-s+1)x_rx_s \implies x_{r+1}x_{s-1}\in UB
\]
Therefore, given $(r,s)\in S$ we have:
\[
r\ge 1 \implies (r-1,s+1)\in S \quad {\rm and}\quad s\ge 1 \implies (r+1,s-1)\in S
\]
By induction, it follows that either $S=\emptyset$ or $S=\{ (r,s)\in\N^2\, |\,  r+s=d\}\subset UB$. If the latter case holds, then $T_{d}^{(0)}\in A\cap UB$. 
However, by the Structure Theorem, $A\cap UB=\{ 0\}$, a contradiction. Therefore, $S=\emptyset$, thus proving part (a).

Assume that $M_{a,b}=UF$ for some $a,b$ with $1\le a<b\le d$ and $a+b=d+1$ and some quadratic form $F\in B$. Then
\[
\textstyle U((d-a+1)F+x_{a-1}x_{b-1})=b(\frac{3}{2}d+1)x_{a-1}x_b \implies x_{a-1}x_b\in UB
\]
contradicting part (a). So part (b) is confirmed. 

Part (c) follows from homogeneity and {\it Proposition\,\ref{quadratic-ideal}(b)}.
Since 
\[
T_{2i}^{(d-2i)}=UT_{2i}^{(d-2i-1)} \quad  (1\le i<\textstyle\frac{d}{2})
\]
and since $T_d^{(0)}\in A$, part (d) follows from part (c). 
\end{proof}

\begin{lemma}\label{d=4m} Assume that $d\in4\Z$. 
\begin{itemize}
\item [{\bf (a)}] If $r,s\in\N$ and $r+s=d/2$ then $x_rx_s\not\in UB$. 
\item [{\bf (b)}] If $a,b\in\N$ is a pair such that $M_{a,b}\in \Theta_d$ then $M_{a,b}\not\in UB$. 
\item [{\bf (c)}] The set $\{T_{2i}^{(\frac{d}{2}-2i)}\,|\, 1\le i\le \frac{d}{4}\}$ is a basis for $\Theta_d$. 
\smallskip
\item [{\bf (d)}] $\Theta_d=(\Theta_d\cap UB)\oplus k\cdot T_{\frac{d}{2}}^{(0)}$ where $\{ T_2^{(\frac{d}{2}-2)} , \hdots , T_{\frac{d}{2}-2}^{(2)}\}$ is a basis for $\Theta_d\cap UB$. 
\end{itemize}
\end{lemma}

The proof of this lemma is almost identical to the proof of {\it Lemma\,\ref{d=2m}} and is therefore omitted. 

\subsection{The Affine Cone over a Rational Normal Curve} Let $k[y,z]\cong k^{[2]}$ with basic fundamental pair $(D,U)=(y\partial_z,z\partial_y)$.
Given the integer $d\ge 1$, let $F_d\subset k[y,z]$ be the vector space of $d$-forms, which is of dimension $d+1$. Then $(D,U)$ restricts to the subring $k[F_d]$ and 
the surface $X_d={\rm Spec}(k[F_d])$ is the affine cone over the rational normal curve in $\PP^2_k$ of degree $d$. 

\begin{lemma}\label{RNC} Given $d\ge 1$, $k[F_d]$ is normal and the minimum number of generators of $k[F_d]$ as a $k$-algebra is $d+1$. Consequently, the rings $k[F_d]$, $d\ge 1$, are pairwise non-isomorphic. 
\end{lemma}

\begin{proof} 
Given $d\ge 1$, define:
\[
\textstyle L=k[y^d,\frac{z}{y}]\cap k[z^d,\frac{y}{z}]\subset {\rm frac}(k[F_d])
\]
Since $k[y^d,\frac{z}{y}]\cong k^{[2]}$ and $k[z^d,\frac{y}{z}]\cong k^{[2]}$ it follows that $L$ is normal. 
Since
\[
y^d\left(\frac{z}{y}\right)^i=y^{d-i}z^i \quad {\rm and}\quad z^d\left(\frac{y}{z}\right)^i=y^iz^{d-i} \quad (0\le i\le d)
\]
we see that $k[F_d]\subseteq L$. For the reverse inclusion, let $y^iz^j\in L$ for some $i,j\in\N$. Then
\[
y^iz^j=(y^d)^r\left( \frac{z}{y}\right)^j=(z^d)^s\left(\frac{y}{z}\right)^i
\]
for some $r,s\in\N$. Therefore:
\[
y^iz^j=y^{dr-j}z^j=(y^d)^{r-1}y^{d-j}z^j\in k[F_d]
\]
It follows that $k[F_d]=L$ and $k[F_d]$ is normal. 

Let $m$ be the minimal number of generators of $k[F_d]$ as a $k$-algebra. Then $m\le \dim_kF_d=d+1$. As discussed above, $k[F_d]\cong k[x_0,\hdots ,x_d]/J$ where $J$ is the ideal generated by the quadratic forms $M_{a,b}$ , $1\le a<b\le d$. 
The Jacobian matrix $\mathcal{J}$ defined by these generators is of dimension ${d\choose 2}\times(d+1)$ and each entry is of the form 
$cx_i$ for some $c\in k$ and $0\le i\le d$. Let $\mathfrak{m}\subset k[x_0,\hdots ,x_d]$ be the maximal ideal $\mathfrak{m}=(x_0,\hdots ,x_n)$, noting that $J\subset\mathfrak{m}$. 
We see that $\mathcal{J}_0:=\mathcal{J} \, ({\rm mod}\, \mathfrak{m})$ is the zero matrix and therefore:
\[
\dim (\mathfrak{m}/\mathfrak{m}^2)=(d+1)-{\rm rank}\,\mathcal{J}_0=d+1
\]
So the dimension of the tangent space to $X_d$ at the point defined by  $\mathfrak{m}$ is $d+1$, meaning that $m\ge d+1$. 
\end{proof}

\subsection{Smooth $SL_2(k)$-Surfaces} Let $k[x_0,x_1,x_2]\cong k^{[3]}$ with basic fundamental pair:
\[
(x_0\partial_{x_1}+x_1\partial_{x_2},2x_1\partial_{x_0}+2x_2\partial_{x_1})
\]
The ring of $SL_2(k)$ invariants is $k[f]$ for $f=2x_0x_2-x_1^2$.
For each integer $d\ge 1$, the action restricts to the affine subring $k[W_d]$, where $W_d\subset k[x_0,x_1,x_2]$ is the vector space of ternary forms of degree $d$. 
Given $\lambda\in k^*$, define
\[
S_{\lambda}(k)=k[x_0,x_1,x_2]/(f-\lambda )
\]
which is the coordinate ring of a smooth Danielewski surface (see {\it Section\,\ref{applications}}), and let 
\[
\pi_{\lambda}:k[x_0,x_1,x_2]\to S_{\lambda}(k)
\]
be the standard projection, which is equivariant. The $SL_2(k)$ action on $S_{\lambda}(k)$ restricts to each subring $\pi_{\lambda}k[W_d]$. 
Let $Q_{\lambda}(k)=\pi_{\lambda}k[W_2]$. It is easy to check that $\pi_{\lambda}k[W_d]\cong S_{\lambda}(k)$ for odd $d$, and 
$\pi_{\lambda}k[W_d]\cong Q_{\lambda}(k)$ for even $d$. Moreover, {\it Proposition\,\ref{Danielewski}} below shows that $S_{\lambda}(k)$ and $Q_{\mu}(k)$ are not isomorphic for all pairs $\lambda ,\mu\in k^*$.  

In order to see that ${\rm Spec}(Q_{\lambda}(k))$ is smooth, define a free $\Z_2=\langle \mu\rangle$-action on $S_{\lambda}(k)$ by $\mu (x_i)=-x_i$, $1\le i\le 3$.  
Then $S_{\lambda}(k)^{\Z_2}=Q_{\lambda}(k)$ and freeness of the action in this case implies that ${\rm Spec}(Q_{\lambda}(k))$ is smooth. 

\subsection{Classification}

Over the field $k=\C$, normal affine $SL_2(\C )$-surfaces were classified by Gizatullin \cite{Gizatullin.71c} and Popov \cite{Popov.73a}.
The list of isomorphism classes for these coordinate rings is comprised of $\C [F_d]$ for $d\ge 1$, $S_1(\C )$ and $Q_1(\C )$. We generalize this classification to all fields of characteristic zero.

\begin{theorem}\label{classification} {\bf (Classification Theorem)} Let $B$ be a normal affine $k$-domain with $\dim_kB=2$ and $B^*=k^*$. 
If $B$ admits a nontrivial $SL_2(k)$-action then $B$ is equivariantly isomorphic to either $k[F_d]$ for some $d\ge 1$, or to $S_{\lambda}(k)$ or $Q_{\lambda}(k)$ for some $\lambda\in k^*$. Among these rings, the unique factorization domains are 
$k[F_1]\cong k^{[2]}$ and $S_{\lambda}(k)$ for $\lambda\in k^*$ such that $\sqrt{-\lambda}\not\in k$. 
\end{theorem}
Note that the equivariance condition in this theorem implies that there are only two equivalence classes of fundamental pairs for $B$, namely, trivial and nontrivial. 

\begin{proof} Let $K={\rm frac}(B)$. Since $B^*=k^*$, $k$ is algebraically closed in $B$. 
Let $(D,U)$ be a nontrivial fundamental pair for $B$ and let $B=\bigoplus_{i\in\Z}B_i$ be the induced grading of $B$. Let $A=\krn D$ and $\Omega  =\krn U$. 
Then $\dim_kA=1$ and $\dim_kA_0=0$. 
Therefore, $A_0$ is an algebraic extension of $k$, so $A_0=k$. Since $ML(B)\subseteq A_0$ we see that $ML(B)=k$.
By {\it Lemma\,\ref{Kolhatkar}}, $A\cong k^{[1]}$, and $\Omega\cong k^{[1]}$ by symmetry. 

Let $d\ge 1$ and $h_d\in \Omega_{-d}$ be such that $\Omega =k[h_d]$. 
Define $h_{d-i}=D^ih_d$, $0\le i\le d$, noting that $h_0\in A_d$ and $A=k[h_0]$. Set $R=k[h_0,h_d]$. 
By {\it Lemma\,\ref{generators}} we have:
\begin{equation}\label{integral}
B=k[h_0,\hdots ,h_d]=R\oplus Rh_1\oplus\cdots\oplus Rh_{d-1}
\end{equation}
Define the surjection $\varphi :k[x_0,\hdots ,x_d]\to B$ by $\varphi (x_i)=h_i$. Then $\varphi$ is equivariant for the basic fundamental pair $(\Delta ,\Upsilon )$ on $k[x_0,\hdots ,x_d]$ and $\krn\varphi$ is $(\Delta ,\Upsilon )$-invariant; see {\it Section\,\ref{basic}}.
Let $m$ be such that $d=2m$ or $d=2m+1$. For each quadratic form $T_{2i}^{(0)}\in \krn\Delta$, $1\le i\le m$, we have $\varphi (T_{2i}^{(0)})\in A=k[h_0]$. 
By homogeneity, there exist $r_i\in k$ and $e_i\in\N$ with $T_{2i}^{(0)}-r_ix_0^{e_i}\in\krn\varphi$. 

Consider the case where $r_i=0$ for each $i$. By $(\Delta ,\Upsilon )$-invariance of $\krn\varphi$ we see that:
\[
(T_2^{(0)},\hdots ,T_d^{(0)})\subset \krn\varphi \implies (\widehat{T}_2,\hdots ,\widehat{T}_d)\subseteq \krn\varphi
\]
By {\it Lemma\,\ref{quadratic-ideal}} it follows that $(\widehat{T}_2,\hdots ,\widehat{T}_d)=\krn\varphi$ and $B\cong k[F_d]$. 

Consider the case where $r_i\ne 0$ for some $i$. Then 
\[
2d-4i=\deg T_{2i}^{(0)} = e_id \implies (2-e_i)d=4i\ge 4 \implies e_i \in\{ 0,1\}
\]
which means that either
\begin{enumerate}
\item $d=2m$, $i=m$ and $T_d^{(0)}-\mu\in\krn\varphi$ for some $\mu\in k^*$, or
\item $d=2m$, $2i=m$ and $T_{\frac{d}{2}}^{(0)}-\gamma x_0\in\krn\varphi$ for some $\gamma\in k^*$. 
\end{enumerate}
Assume that $d=2m$ for odd $m\ge 3$. Then $(\widehat{T}_2,\hdots ,\widehat{T}_{d-2}, \widehat{T}_d-\mu)\subseteq\krn\varphi$. 
By {\it Lemma\ref{quadratic-ideal}}, the minors $M_{\frac{d}{2},\frac{d}{2}+1}$, $M_{\frac{d}{2}-1,\frac{d}{2}}$ and $M_{\frac{d}{2}-1,\frac{d}{2}+1}$ are linear combinations of the vertices of $\widehat{T}_2,\hdots ,\widehat{T}_d$. We have:
\[
\deg M_{\frac{d}{2},\frac{d}{2}+1}=0 \,\, ,\,\, 
\deg M_{\frac{d}{2}-1,\frac{d}{2}+1}=2 \,\, ,\,\,
\deg M_{\frac{d}{2}-1,\frac{d}{2}}=4
\]
Therefore, $M_{\frac{d}{2}-1,\frac{d}{2}+1}$ and $M_{\frac{d}{2}-1,\frac{d}{2}}$ do not involve $\widehat{T}_d=T_d^{(0)}$, which implies:
\[
M_{\frac{d}{2}-1,\frac{d}{2}+1} \, , \,  M_{\frac{d}{2}-1,\frac{d}{2}}\in\krn\varphi
\]
In addition, by {\it Lemma\,\ref{d=2m}}, there exist $c_1,\hdots ,c_{\frac{d}{2}}\in k$ such that
\[
M_{\frac{d}{2},\frac{d}{2}+1}=\sum_{i=1}^{\frac{d}{2}} c_iT_{2i}^{(d-2i)}
\]
and $c_{\frac{d}{2}}\ne 0$. Therefore, $M_{\frac{d}{2},\frac{d}{2}+1}\equiv c_{\frac{d}{2}}\mu$ modulo $\krn\varphi$. 
From the syzygy
\[
\textstyle
(\frac{d}{2}-1)x_{\frac{d}{2}-1}M_{\frac{d}{2},\frac{d}{2}+1} - \frac{d}{2}x_{\frac{d}{2}}M_{\frac{d}{2}-1,\frac{d}{2}+1} + (\frac{d}{2}+1)x_{\frac{d}{2}+1}M_{\frac{d}{2}-1,\frac{d}{2}}=0
\]
it follows that $x_{\frac{d}{2}+1}c_{\frac{d}{2}}\mu\in\krn\varphi$, which is a contradiction. So the case $d=2m$ for odd $m\ge 3$ cannot occur. 

Assume that $d=4m$ for $m\ge 2$. 
In this case, $(\widehat{T}_2,\hdots ,\widehat{T}_{\frac{d}{2}}-\gamma \hat{x}_0, \hdots ,\widehat{T}_{d-2}, \widehat{T}_d-\mu)\subseteq\krn\varphi$.\footnote{
Here, the cable $\hat{x}_0$ is $\{ x_0, Ux_0, U^2x_0, \hdots ,U^dx_0\}$.}
By {\it Lemma\ref{quadratic-ideal}}, the minors $M_{\frac{d}{4},\frac{d}{4}+1}$, $M_{\frac{d}{4}-1,\frac{d}{4}}$ and $M_{\frac{d}{4}-1,\frac{d}{4}+1}$ are linear combinations of the vertices of $\widehat{T}_2,\hdots ,\widehat{T}_{\frac{d}{2}}$. We have:
\[
\deg M_{\frac{d}{4},\frac{d}{4}+1}=d \,\, ,\,\, 
\deg M_{\frac{d}{4}-1,\frac{d}{4}+1}=d+2 \,\, ,\,\,
\deg M_{\frac{d}{4}-1,\frac{d}{4}}=d+4
\]
Since $d>4$ we see that $d$ does not divide $d+2$ or $d+4$. Therefore, $M_{\frac{d}{4}-1,\frac{d}{4}+1}$ and $M_{\frac{d}{4}-1,\frac{d}{4}}$ do not involve $T_{\frac{d}{2}}^{(0)}$ or $T_d^{(0)}$, which implies:
\[
M_{\frac{d}{4}-1,\frac{d}{4}+1} \, , \, M_{\frac{d}{4}-1,\frac{d}{4}}\in\krn\varphi
\]
In addition, by {\it Lemma\,\ref{d=4m}}, there exist $c_1,\hdots ,c_{\frac{d}{4}}\in k$ such that
\[
M_{\frac{d}{4},\frac{d}{4}+1}=\sum_{i=1}^{\frac{d}{4}} c_iT_{2i}^{(\frac{d}{2}-2i)}
\]
and $c_{\frac{d}{4}}\ne 0$. Therefore, $M_{\frac{d}{4},\frac{d}{4}+1}\equiv c_{\frac{d}{4}}\gamma x_0$ modulo $\krn\varphi$. 
From the syzygy
\[
\textstyle 
(\frac{d}{4}-1)x_{\frac{d}{4}-1}M_{\frac{d}{4},\frac{d}{4}+1} - \frac{d}{4}x_{\frac{d}{4}}M_{\frac{d}{4}-1,\frac{d}{4}+1} + (\frac{d}{4}+1)x_{\frac{d}{4}+1}M_{\frac{d}{4}-1,\frac{d}{4}}=0
\]
it follows that $c_{\frac{d}{4}}\gamma x_0x_{\frac{d}{4}-1}\in\krn\varphi$, which is a contradiction. So the case $d=4m$ for $m\ge 2$ cannot occur.
Therefore, either $d=2$ or $d=4$. 

Suppose that $d=2$. Then $T_2^{(0)}-\mu\in\krn\varphi$ and since both $(T_2^{(0)}-\mu )$ and $\krn\varphi$ are height-one primes of $k^{[3]}$ we see that 
$\krn\varphi = (T_2^{(0)}-\mu)$. Therefore, $B\cong k[x_0,x_1,x_2]/(2x_0x_2-x_1^2-\mu )=S_{\mu}(k)$.

Suppose that $d=4$. Then $(\widehat{T}_2-\gamma x_0,\widehat{T}_4-\mu )\subseteq\krn\varphi$. By direct calculation, we find 
\[
T_2^{(0)}=2x_0x_2-x_1^2\,\, ,\,\, T_2^{(1)}=4(3x_0x_3-x_1x_2)\,\, ,\,\, T_2^{(2)}=24(2x_0x_2+x_1x_3-x_2^2) 
\]
and
\[
T_4^{(0)}=2x_2x_2-2x_1x_3+x_2^2
\]
as well as the syzygy:
\[
x_0T_2^{(2)}-6x_1T_2^{(1)}+24x_2T_2^{(0)}=24x_0T_4^{(0)}
\]
Since $Ux_0=4x_1$ and $U^2x_0=24x_2$ it follows that:
\[
x_0(24\gamma x_2) - 6x_1(4\gamma x_1)+24x_2(\gamma x_0)-24x_0\mu \in\krn\varphi
\]
From this, together with the fact that $T_2^{(0)}-\gamma x_0\in\krn\varphi$ we obtain:
\[
\gamma T_2^{(0)}-\mu x_0\in\krn\varphi  \implies \gamma^2=\mu
\]
Therefore, $\gamma ,\mu\ne 0$. It can be checked directly that the ideal of relations for $Q_{\lambda}(k)$ is:
\[
(\widehat{T}_2-2\lambda\hat{x_0} , \widehat{T}_4-4\lambda^2)\subset k[x_0,x_1,x_2,x_3,x_4]
\]
Therefore, $B\cong Q_{\gamma}(k)$. 

It remains to show that, for each $\lambda\in k^*$, $B=Q_{\lambda}(k)$ is not a UFD. 
We have $h_i=\varphi (x_i)$ for $0\le i\le 4$. The element $h_0$ is irreducible, since it generates the kernel of $D$. In addition:
\[
0=\varphi (T_2^{(0)}-2\lambda x_0)=2h_2h_0-h_1^2-2\lambda h_0 \implies h_1^2\in h_0B
\]
Suppose that $h_1=h_0f$ for some $f\in B$. Then $Dh_1=h_0Df\in k\cdot h_0$ implies $Df\in k^*$ and $D$ has a slice, which is a contradiction. 
So $h_1\not\in h_0B$ and $B$ is not a UFD. 

This completes the proof. 
\end{proof}

Note that the number of isomorphism classes represented by the rings $S_{\lambda}(k)$ and $Q_{\lambda}(k)$, $\lambda\in k^*$, depends on the ground field $k$. If $k$ is algebraically closed, then $S_{\lambda}(k)\cong S_1(k)$ for all $\lambda\in k^*$. 
If $k=\R$ then for $\lambda\in \R^*$ either $S_{\lambda}(\R )\cong S_1(\R )$ or $S_{\lambda}(\R )\cong S_{-1}(\R )$, where $S_1(\R )$ is a UFD and $S_{-1}(\R )$ is not a UFD. 
If $k=\Q$, there are infinitely many isomorphism classes for $S_{\lambda}(\Q )$, $\lambda\in\Q^*$. 

%%%%%%%%%%%%%%%%%%%%%%%%%%%%%%%%%%%%%%%%%%%%%%%%%%%%%%%%%%%%%%%%%%%%%%%%%%%%%%

\section{Further Applications of the Structure Theorem}\label{applications}
\subsection{Fundamental pairs for $B^{[1]}$}
Let $B$ be an affine $k$-domain with fundamental pair $(D,U)$ and let $B[t]\cong B^{[1]}$. 
Any fundamental pair $(D',U')$ for $B[t]$ which restricts to $(D,U)$ on $B$ is called an {\bf extension} of $(D,U)$ to $B[t]$. 
An extension $(D',U')$ of $(D,U)$ is {\bf trivial} if there exists $s\in B[t]$ such that $B[t]=B[s]$ and $D's=U's=0$.

\begin{proposition}\label{trivial-ext} 
Let $B$ be an affine $k$-domain and $B[t]\cong B^{[1]}$. Given a fundamental pair $(D,U)$ for $B$, every extension of $(D,U)$ to $B[t]$ is trivial.
\end{proposition}

\begin{proof} Let $(D_0,U_0)$ be the trivial extension of $(D,U)$ defined by $D_0t=U_0t=0$ and let 
$(D',U')$ be any extension of $(D,U)$ to $B[t]$.
By \cite{Freudenburg.17}, Principle 6, 
$D't\in B$, $U't\in B$ and $[D_0,\frac{d}{dt}]=[U_0,\frac{d}{dt}]=0$. If $f=D't$ and $g=U't$ then:
\[
D'=D_0+f\frac{d}{dt} \quad {\rm and}\quad U'=U_0+g\frac{d}{dt}
\]
Set $E=[D,U]$, $E_0=[D_0,U_0]$ and $E'=[D',U']$. By direct calculation we find that:
\[
E'=E_0+h\frac{d}{dt}\quad {\rm where}\quad h=Dg-Uf
\]
Therefore, 
\[
-2D'=[D',E']=-2D_0+(Dh-Ef)\frac{d}{dt} \implies Dh=Ef-2f
\]
and:
\[
2U'=[U',E']=2U_0+(Uh-Eg)\frac{d}{dt} \implies Uh=Eg+2g
\]
Let $f=\sum_df_d$ be the decomposition of $f$ into homogeneous summands. Then:
\[
Ef-2f=\sum_d(d-2)f_d\in DB
\]
Since $DB$ is a graded submodule of $B$, it follows that $f_d\in DB$ for each $d\ne 2$. In addition, 
$B_2=\Omega_2\oplus DB_0=DB_0$ by the Structure Theorem, so $f_2\in DB$ as well. Therefore, $f\in DB$.
Choose $p\in B$ be such that $Dp=f$. Then $D'(t-p)=0$. So we may assume, with no loss of generality, that $f=0$. 

Suppose that $g\ne 0$. Let $g=\sum_{d\le \gamma}g_d$ be its homogeneous decomposition, where $\gamma =\deg g$.
Since $f=0$ we see that $Dg=Dg-Uf=h$ and $D^2g=Dh=Ef-2f=0$. Since $D$ is homogeneous, $Dg_d\in A_d$ for each $d\le\gamma$. By the Structure Theorem, $d\ge -2$ when $g_d\ne 0$. In particular, $\gamma\ge -2$. 

If $\gamma\ne -2$, define $t'=(\gamma +2)t-h$. We have
\[
U't'=(\gamma +2)g-Uh=\gamma g+2g-Eg-2g=\sum_{d<\gamma}(\gamma -d)g_d
\]
and since $Dh=0$ we also have $D't'=0$. 
Therefore, when $g\ne 0$ and $\deg g>-2$ we can replace $g=Ut$ by $g'=U't'$, where $-2\le d<\deg g$ when $(g')_d\ne 0$, while preserving the properties $D't'=D't=0$ and $B[t']=B[t]$. 
By induction, we can assume that $g\in B_{-2}$. In this case, $h=Dg\in A_0\cap DB_{-2}=\{ 0\}$ by the Structure Theorem, so $g\in A_{-2}=\{ 0\}$. 

Therefore, we can find $s\in B[t]$ with $B[t]=B[s]$ and $D's=U's=0$. 
\end{proof}

\subsection{Fundamental Pairs for $R^{[2]}$ over $R$}

\begin{proposition}\label{R[X,Y]} Let $R$ be an affine UFD over $k$ and $B=R[X,Y]\cong R^{[2]}$. 
Every nontrivial fundamental pair $(D,U)\in {\rm LND}_R(B)^2$ is conjugate by an element of ${\rm Aut}_R(B)$ to the fundamental pair $(X\partial_Y, Y\partial_X)$.
\end{proposition}

\begin{proof} Let $K={\rm frac}(R)$ and let $(D_K,U_k)$ be the extension of $(D,U)$ to $K[X,Y]\cong_KK^{[2]}$. By the Classification Theorem, there exist $P,Q\in K[X,Y]$ such that 
$K[P,Q]=K[X,Y]$ and $(D_K,U_K)=(P\partial_Q , Q\partial_P)$. Since $K[X,Y]=K[P,Q]$ there exist $v,w\in K^*$ such that, if $P'=vP$ and $Q'=wQ$, then $P',Q'\in B$ and $P',Q'$ are irreducible. 
Observe that, if $D'=(w^{-1}v)D_K$ and $U'=(wv^{-1})U_K$, then $(D',U')$ is a fundamental pair for $K[X,Y]$. In addition, $D_KQ=P$ and $U_KP=Q$ implies that $D'Q'=P'$ and $U'P'=Q'$. 
Therefore, $(D',U')$ restricts to $B$. 

Let $A=\krn D$ and $\Omega =\krn U$. Since $K[P']=\krn D'=\krn D_K$ and $K[Q']=\krn U'=\krn U_K$ it follows that $A=R[P']$ and $\Omega = R[Q']$; see \cite{Freudenburg.17}, 4.1. Therefore, $A_0=A\cap\Omega=R$ and $A=A_0^{[1]}$. 
In addition, $U'P'=Q'$ implies $\deg P'=\deg_UP'=1$. By {\it Lemma\,\ref{generators}}, $B=A_0[P',Q']=R[P',Q']$ and $(D,U)=(P'\partial_{Q'},Q'\partial_{P'})$. 
\end{proof}

\subsection{Fundamental Pairs for Danielewski Surfaces}
Let $\varphi (T)\in k[T]\cong k^{[1]}$ be nonconstant and let $k[X,Y,Z]\cong k^{[3]}$. The surface defined by $B=k[X,Y,Z]/(XY-\varphi (Z))$ is called a special Danielewski surface. 
It is shown in \cite{Daigle.04}, Proposition 2.3(a) that $B$ is normal. 
By \cite{Daigle.04}, Lemma 2.10, the polynomial $\varphi$ is uniquely determined by $B$ up to a $k$-automorphism of $k[T]$ and multiplication by a unit. In particular, the degree of $\varphi$ is uniquely determined by $B$
and we have that $B\cong k^{[2]}$ if and only if $\deg_T\varphi (T)=1$. Note that, for all $\lambda\in k^*$, $S_{\lambda}(k)$ defines a special Danielewski surface. 
 
 \begin{proposition}\label{Danielewski} Let $B=k[X,Y,Z]/(XY-\varphi (Z))$ where $\deg_T\varphi\ge 2$. 
If $B$ admits a nontrivial fundamental pair then $\deg_T\varphi (T)=2$ and $B$ is equivariantly isomorphic to $k[F_2]$ or 
$S_{\lambda}(k)$ for some $\lambda\in k^*$. Moreover, for all $\lambda\in k^*$, ${\rm Spec}(Q_{\lambda}(k))$ is not isomorphic to a special Danielewski surface. 
\end{proposition} 

\begin{proof} Since $B^*=k^*$, {\it Theorem\,\ref{classification}} implies that 
$B$ is equivariantly isomorphic to one of the rings $k[F_d]$ for $d\ge 1$, or $S_{\lambda}(k)$ or $Q_{\lambda}(k)$ for some $\lambda\in k^*$. 

Suppose that $B\cong k[F_d]$ for some $d\ge 1$. Since ${\rm Spec}(B)$ is a hypersurface in $\A_k^3$, {\it Lemma\,\ref{RNC}} implies $d\le 2$. 
In addition, $\deg_T\varphi (T)\ge 2$ implies $B\ne k^{[2]}=k[F_1]$. 
Therefore, $d=2$ and $B\cong k[F_2]$, which implies $\deg_T\varphi (T)=2$. 

Suppose that $B\cong S_{\lambda}(k)$ for some $\lambda\in k^*$. Then $\deg_T\varphi (T)=2$.

Suppose that $B\cong Q_{\lambda}(k)$ for some $\lambda\in k^*$. Recall that 
\[
S_{\lambda}(k)=k[x_0,x_1,x_2]/(2x_0x_2-x_1^2-\lambda )=k[\bar{x}_0,\bar{x}_2,\bar{x}_3]
\]
and $Q_{\lambda}(k)=k[h_4,h_3,h_2,h_1,h_0] \subset S_{\lambda}(k)$ where
\[
h_4=\bar{x}_2^2\,\, ,\,\, h_3=\bar{x}_1\bar{x}_2\,\, ,\,\, h_2=3\bar{x}_0\bar{x}_2-\lambda\,\, ,\,\, h_1=\bar{x}_0\bar{x}_1 \,\, ,\,\, h_0=\bar{x}_0^2
\]
and where $\krn\delta =k[h_0]$ and $\delta h_1=h_0$ for the fundamental pair $(\delta ,\upsilon )$ on $B$. By \cite{Daigle.04}, Lemma 2.8, there exists $\psi\in k[T]$ and $y\in B$ such that $Q_{\lambda}(k)=k[h_0,h_1,y]$ and $h_0y=\psi (h_1)$. 
Therefore, in $k[x_0,x_1,x_2]$ we have:
\[
\psi (x_0x_1)\in (x_0^2 , 2x_0x_2-x_1^2-\lambda) \implies \psi (0)+\psi^{\prime}(0)x_0x_1\in (x_0^2,2x_0x_2-x_1^2-\lambda)
\]
\[
\implies \psi (0)=\psi^{\prime}(0)=0 \implies \psi\in T^2k[T]
\]
But then $Q_{\lambda}(k)$ is singular, a contradiction. So this case cannot occur. 
\end{proof}

%%%%%%%%%%%%%%%%%%%%%%%%%%%%%%%%%%%%%%%%%%%%%%%%%%%%%%%%%%%%%

\section{UFDs of Dimension Three} 
\subsection{Certain UFDs with a fundamental pair}

This section classifies UFDs of dimension three over $k$ which admit a certain type of fundamental pair. The following technical lemma is required. 

\begin{lemma}\label{technical} {\rm (\cite{Freudenburg.17}, Corollary 5.42)} Let $B$ be a UFD over $k$ and let $\delta\in {\rm LND}(B)$ be nonzero. Suppose that $S\subset B$ is a factorially closed $k$-subalgebra 
such that $k\ne S\cap\krn\delta\ne S$. If $S\cong k^{[2]}$, then there exists $w\in S$ such that $S\cap\krn\delta=k[w]$ and $S=k[w]^{[1]}$. 
\end{lemma}

\begin{proposition}\label{threefold} Let $B$ be an affine UFD over $k$ which admits a nontrivial fundamental pair $(D,U)$ such that
$A_0\ne k$ and $A\cong k^{[2]}$, where $A=\krn D$ and $A_0=A\cap\krn U$. Then either
\begin{enumerate}
\item $B=k[X,Y,Z]= k^{[3]}$ and $(D,U)$ is equivalent to $(Y\frac{\partial}{\partial Z}\, ,\, Z\frac{\partial}{\partial Y})$; or
\smallskip
\item $B=k[X,Y,Z,W]/(2XZ-Y^2-P(W))$ for some $P(T)\in k[T]\cong k^{[1]}$ and 
$(D,U)$ is equivalent to the fundamental pair induced by 
$(X\partial_Y+Y\partial_Z \, ,\,  2Y\partial_X+2Z\partial_Y)$
on $k[X,Y,Z,W]= k^{[4]}$.
\end{enumerate}
In addition, $ML(B)=k$. 
\end{proposition}

\begin{proof} 
Note that $\dim_kB=\dim_kA+1=3$, and that $B^*=A^*=k^*$. Using $\delta =U$ and $S=A$ in {\it Lemma\,\ref{technical}}, we find that there exist $f,g\in B$ such that $A_0=k[f]$ and $A=k[f,g]=A_0^{[1]}$. 
We may assume that $g\in A_d$ for some $d\ge 1$. Define $h_0=g$ and:
\[
h_i=\begin{cases} U^i\left( g^{i/d}\right) & i\in d\Z \\ U^i\left( g^{[i/d]+1}\right) & i\not\in d\Z \end{cases} \quad (i\ge 1)
\]
By {\it Lemma\,\ref{generators}}, $B=A_0[h_0,\hdots ,h_d]$. Let $K=k(f)$ and: 
\[
B_K=K\otimes_{k[f]}B=K[h_0,\hdots ,h_d]
\]
Then $B_K$ is an affine UFD of dimension two over $K$ with $B_K^{\ast}=K^*$. The pair $(D,U)$ extends to a fundamental pair $(D_K,U_K)$ on $B_K$. 
By {\it Theorem\,\ref{classification}} either $d=1$ and $B_K\cong K^{[2]}$, or $d=2$ and $B_K\cong S_{\lambda}(K)$ for some $\lambda\in K^*$ 
where the isomorphism is equivariant for the (essentially unique) nontrivial fundamental pair on $S_{\lambda}(K)$. 

Assume that $d=1$. Then $B=k[f,h_0,h_1]\cong k^{[3]}$. Relabel $(X,Y,Z)=(f,h_0,h_1)$. 
Since $Dh_1=f_0$, $Dh_0=Df=0$ and $Uh_0=h_1$, $Uh_1=Uf=0$ we see that $(D,U)=(Y\partial_Z , Z\partial_Y)$. So statement (1) of the proposition holds if $d=1$. 

Assume that $d=2$. Then $B=k[f,h_0,h_1,h_2]$ and $B_K=K[h_0,h_1,h_2]$ where $2h_0h_2-h_1^2=\lambda$ for some nonzero $\lambda\in K^*\cap B=k[f]$. Note that we used equivariance of the isomorphism to get this equation. 
Let $P(T)\in k[T]$ be such that $\lambda =P(f)$. Then:
\[
B\cong k[X,Y,Z,W]/(2XZ-Y^2-P(W))
\]
Since 
\[
Df=Dh_0=0\,\, ,\,\, Dh_1=2h_0\,\, ,\,\, Dh_2=2h_1 \quad {\rm and}\quad Uh_0=h_1\,\, ,\,\, Uh_1=h_2\,\, ,\,\, Uh_2=Uf=0
\]
we see that $(D,U)$ is equivalent to the fundamental pair induced by 
$(X\partial_Y+Y\partial_Z \, ,\,  2Y\partial_X+2Z\partial_Y)$ on $k[X,Y,Z,W]= k^{[4]}$.
So statement (2) of the proposition holds when $d=2$. 

In case (1), $ML(k^{[3]})=k$ by considering partial derivatives. In case (2), $ML(B)\subset A_0=k[W]$, so in order to show $ML(B)=k$ it suffices to find $\Delta\in {\rm LND}(B)$ with $\Delta (W)\ne 0$. We can take
$\Delta$ to be the derivation of $B$ induced by $P'(W)\partial_Z+2X\partial_W$ on $k[X,Y,Z,W]$. 
\end{proof}

\begin{example}\label{Russell} {\rm Let $B$ be the coordinate ring of the the Russell cubic threefold $X$, the hypersurface in $\A_k^4$ defined by $x+x^2y+z^2+t^3=0$. 
Suppose that $(D,U)$ is a nontrivial fundamental pair for $B$ with $A_0=\krn D\cap\krn U$. 
It is well known that $ML(B)=k[x]$ so $k[x]\subset A_0$ and $A_0\ne k$ (see \cite{Makar-Limanov.96}). In addition, \cite{Freudenburg.17}, Corollary 9.10 shows that the kernel of any nonzero locally nilpotent derivation of $B$ equals 
$k[x,P]$ for some $P\in B$. 
Therefore, $B$ satisfies the hypotheses of {\it Proposition\,\ref{threefold}}. But then $ML(B)=k$, a contradiction. Therefore, $X$ has no nontrivial $SL_2(k)$-action. 
See \cite{Dubouloz.Moser-Jauslin.Poloni.14} for a description of the group of automorphisms of $X$. 
}
\end{example}

\subsection{Fundamental Pairs for $R^{[3]}$ over $R$}
The theorem of Kraft and Popov cited above shows that, over an algebraically closed field $k$, any nontrivial action of $SL_2(k)$ on the polynomial ring $B=k^{[3]}$ is linearizable. 
In particular, the action on $B$ is induced by one of the representations $V_0\oplus V_1$ or $V_2$. 
The following theorem generalizes this result. Its proof gives a new proof of the Kraft-Popov theorem for $SL_2(k)$. 

\begin{theorem}\label{Kraft-Popov} Let $R$ be an affine UFD over $k$ and $B=R[x_0,x_1,x_2]\cong R^{[3]}$. 
Let $(D,U)\in {\rm LND}_R(B)^2$ be a nontrivial fundamental pair for $B$. 
Then $(D,U)$ is equivalent over $R$ to either
\[
(x_0\partial_{x_1} , x_1\partial_{x_0}) \quad or \quad 
(x_0\partial_{x_1}+x_1\partial_{x_2} , 2x_1\partial_{x_0}+2x_2\partial_{x_1}) .
\]
\end{theorem}

\begin{proof} We first prove the result in the case $R=k$. 

Let $A=\krn D$ and $\Omega =\krn U$. By {\it Theorem\,\ref{Miyanishi}(a)} there exist $f,g\in A$ with $A=k[f,g]\cong k^{[2]}$. 
Since $A$ is a graded subring of $B$ we may choose $f,g$ to be homogeneous. Let $d,e\in\N$ be such that $f\in A_e$ and $g\in A_d$. 

If $d,e>0$ then the Structure Theorem implies $f,g\in I_1=A\cap DB$. Since $fA+gA$ is a maximal ideal of $A$ we see that $fA+gA=I_1$. 
But $I_1$ is a principal ideal of $A$ by {\it Theorem\,\ref{Miyanishi}(b)}, which gives a contradiction. Therefore, either $d=0$ and $e>0$, or $d>0$ and $e=0$. 
We may assume without loss of generality that $e=0$ and $d>0$. Then $A_0=k[f]\cong k^{[1]}$ and $A=A_0[g]\cong k^{[2]}$. By {\it Theorem\,\ref{threefold}} it follows that 
either
\begin{itemize}
\item [(i)] $B=k[X,Y,Z]$ and $(D,U)=(X\partial_Y, Y\partial_X)$; or
\smallskip
\item [(ii)] $B=k[X,Y,Z,W]/(2XZ-Y^2-P(W))$ for some $P(T)\in k[T]\cong k^{[1]}$ and $(D,U)$ is induced by 
$(X\partial_Y+Y\partial_Z  , 2Y\partial_X+2Z\partial_Y)$
on $k[X,Y,Z,W]= k^{[4]}$.
\end{itemize}
In case (i) there is nothing further to show, so assume that case (ii) holds. 
By {\it Theorem\,\ref{Gupta}}, the polynomial $Y^2-P(W)$ is a variable of $k[Y,W]\cong k^{[2]}$, which implies $\deg_TP(T)=1$. 
Let $x,y,z,w\in B$ be the images of $X,Y,Z,W$, respectively. Since $\deg P_T(T)=1$ we see that $w\in k[x,y,z]$. Therefore, 
$B=k[x,y,z]$ and $(D,U)=(x\partial_y+y\partial_z \, ,\,  2y\partial_x+2z\partial_y)$.

So the theorem holds in the case $R=k$. 

For the general case, let $L={\rm frac}(R)$. Then $(D,U)$ extends to a fundamental pair $(D_L,U_L)$ for $B_L=L[x_0,x_1,x_2]$. By the case for fields, there exist 
$X,Y,Z\in B_L$ such that $B_L=L[X,Y,Z]$ and $(D_L,U_L)$ equals either
$(X\partial_Y, Y\partial_X)$ or $(X\partial_Y+Y\partial_Z , 2Y\partial_X+2Z\partial_Y)$

We may assume that $X,Y,Z\in B$ and that $X,Y,Z$ are irreducible, hence prime, in $B$. 

Let $(I_L)_n$, $n\ge 0$, be the image ideals for $D_L$. Since $D_LY=X$ in either case, we see that $(I_L)_n=X^nA$ for each $n\ge 0$. Therefore, given 
$n\ge 0$ and $f\in I_n$ there exists nonzero $r\in R$ with $rf\in X^nA$. Since $X$ is prime in $A$ and $r\not\in XA$ ($\deg r=0$ while $\deg X>0$), it follows that 
$f\in X^nA$. So $I_n=X^nA$ for each $n\ge 0$. Consequently:
\[
\mathcal{F}_n=A\cdot U^n(X^n)\oplus\mathcal{F}_{n-1}\quad \forall n\ge 1
\]
Note that $U$ restricts to the subring $\tilde{B}:=R[X,Y,Z]$ in both cases. Therefore, $U^n(X)\in \tilde{B}$ for all $n\ge 0$, 
and it follows by induction that $\mathcal{F}_n\subset \tilde{B}$ for all $n\ge 0$. Since $B=R[{\mathcal{F}_N}]$ for some $N\ge 1$ we conclude that $\tilde{B}=B$.
\end{proof}

The following corollary to {\it Proposition\,\ref{Kraft-Popov}} extends Panyushev's theorem to fields of characteristic zero for certain kinds of $SL_2(k)$ actions.
\begin{corollary}\label{Panyushev} Let $B=k[x_0,x_1,x_2,x_3]\cong k^{[4]}$ and suppose that $(D,U)$ is a nontrivial fundamental pair for $B$. If $A_0$ contains a variable of $B$ then $(D,U)$ is equivalent to either 
\[
(x_0\partial_{x_1} , x_1\partial_{x_0}) \quad or \quad 
(x_0\partial_{x_1}+x_1\partial_{x_2} , 2x_1\partial_{x_0}+2x_2\partial_{x_1}) .
\]
\end{corollary}

\begin{proof} Assume that $f\in A_0$ is a variable of $B$ and set $R=k[f]$. Then $B=R^{[3]}$ and $(D,U)\in {\rm LND}_R(B)^2$, so {\it Theorem\,\ref{Kraft-Popov}} gives the desired conclusion. 
\end{proof}

\subsection{A cancellation theorem}

Combining {\it Proposition\,\ref{threefold}} with Panyushev's theorem gives the following cancellation property.

\begin{theorem}\label{cancellation} Assume that $k$ is algebraically closed. Let $X$ be an affine threefold over $k$ such that $X\times\A_k^1\cong\A_k^4$. If $X$ admits a nontrivial $SL_2(k)$-action then $X\cong\A_k^3$. 
\end{theorem}

\begin{proof} 
Let $B$ be an affine $k$-domain isomorphic to $k^{[4]}$ and let $t\in B$ and $R\subset B$ a subalgebra such that $B=R[t]\cong R^{[1]}$. 
Then $R$ is a smooth affine UFD of dimension three over $k$. 

Assume that $(D,U)$ is a nontrivial fundamental pair for $R$. Set $A=\krn D$ and $A_0=A\cap\krn U$. 
Extend $(D,U)$ trivially to the fundamental pair $(D',U')$ on $B$ and set:
\[
A'=\krn D'=A[t] \quad {\rm and}\quad A_0'=A'\cap\krn U'=A_0[t]
\]
By Panyushev's theorem, the $SL_2(k)$-action on $B$ is given by a representation. There exist exactly four such nontrivial representations, namely:
\[
V_1\oplus V_1\,\, ,\,\, V_3\,\, ,\,\, V_0^2\oplus V_1\,\, ,\,\, V_0\oplus V_2
\]
In the first two of these, $A_0'\cong k^{[1]}$ and in the latter two, $A_0'\cong k^{[2]}$. Consider each case. 

Case 1: $A_0'\cong k^{[1]}$. Then $A_0=k$ and $A_0'=k[t]$. For each representation $V_1\oplus V_1$ and $V_3$, 
there exists $h\in A_0'$ with a singular fiber such that $A_0'=k[h]$. Since $B=R[t]$ we see that every fiber of $t$ is smooth, and the equality $k[h]=k[t]$ gives a contradiction. So this case cannot occur.

Case 2: $A_0'\cong k^{[2]}$. Since $A_0'=A_0[t]\cong A_0^{[1]}$ it follows that $A_0\ne k$. For each representation $V_0^2\oplus V_1$ and $V_0\oplus V_2$ we have $A'\cong k^{[3]}$. Since $t\in A'$ we see that $A[t]\subset A'$. Since $A$ is algebraically closed in $R$, $A[t]$ is algebraically closed in $B=R[t]\cong R^{[1]}$, so $A[t]=A'$. By the Cancelation Theorem for Surfaces, $A\cong k^{[2]}$; see \cite{Fujita.79,Miyanishi.Sugie.80}. 
By {\it Proposition\,\ref{threefold}}, it follows that either $R\cong k^{[3]}$, in which case there is nothing further to show, or 
\[
R=k[X,Y,Z,W]/(2XZ-Y^2-P(W))=k[x,y,z,w]
\]
for some nonzero $P(T)\in k[T]\cong k^{[1]}$, where $x,y,z,w\in R$ denote the images of $X,Y,Z,W$, respectively, and where 
$(D,U)$ is equivalent to the fundamental pair induced by 
\[
(X\partial_Y+Y\partial_Z \, ,\,  2Y\partial_X+2Z\partial_Y)
\]
on $k[X,Y,Z,W]= k^{[4]}$.
Using the van den Essen Kernel Algorithm (see {\it Section\,\ref{prelims}}), we find that $A=k[x,w]\cong k^{[2]}$ and $A_0=k[w]\cong k^{[1]}$. So
$A'=k[x,w,t]\cong k^{[3]}$ and $A_0'=k[w,t]$. Note that $x$ is prime in $A'$, and since $A'$ is factorially closed in $B$, $x$ is prime in $B$. 

Consider $V_0^2\oplus V_1$. In this case, $B=A_0'^{[2]}$. So $B=k[w,t]^{[2]}$ implies $t$ is a variable of $B$, and $R=B/tB\cong k^{[3]}$.

Consider $V_0\oplus V_2$. In this case, the plinth ideal $A'\cap D'B=vA'$ for some $v\in A'$ which is a variable of $B$. Since $Dy=x\in A'$ we see that $x\in vB$. Therefore, $xB=vB$ since $x$ is prime in $B$.
It follows that, if $S=k[Y,W]/(Y^2+P(W))$, then:
\[
k^{[3]}=B/vB=B/xB=S[z,t]\cong S^{[2]}
\]
By the Cancelation Theorem for Curves, $S\cong k^{[1]}$.
By the Epimorphism Theorem, $Y^2+P(W)$ is a variable of $k[Y,W]$, which implies that
$\deg_TP(T)=1$ and $R\cong k^{[3]}$. See \cite{Abhyankar.Eakin.Heinzer.72,Abhyankar.Moh.75,Suzuki.74}. 
\end{proof} 

\begin{remark} {\rm This result is similar in spirit to one step in the proof of the Cancelation Theorem for Surfaces ({\it op.cit.}) due to Fujita, Miyanishi and Sugie.
Working over an algebraically closed field $k$ of charactersitic zero, 
they first show that, if $X$ is any factorial affine surface over $k$ with trivial units admitting a nontrivial $\G_a$-action, then $X\cong\A_k^2$. 
The more difficult part is to show that, if $X\times\A_k^n\cong\A_k^{n+2}$ for some surface $X$ and $n\ge 1$, then $X$ admits a nontrivial $\G_a$-action.  }
\end{remark}

%%%%%%%%%%%%%%%%%%%%%%%%%%%%%%%%%%%%%%%%%%%%%%%%%%%%%%%%%%%%%%%%%%%%%%%%%%%%%%

\section{Extensions of Fundamental Derivations}\label{extended}

This section considers two kinds of extensions of an affine $k$-domain $B$ with fundamental pair $(D,U)$. The first kind extends $(D,U)$ to the pair 
$(D+X\partial_Y,U+Y\partial_X)$ on $B[X,Y]\cong B^{[2]}$. The second kind extends $D$ to $B[t]\cong B^{[1]}$ by $Dt=a\in B$, where $Da=Ua=0$. 
We first establish the underlying properties of extensions of locally nilpotent derivations of a certain type in {\it Proposition\,\ref{main2}} below.

{\it Proposition\,\ref{isomorphism}} shows that $\krn D$ is finitely generated, and this forms part (a) of the Structure Theorem. 
Note that no part of the Structure Theorem or its consequences are used in sections {\it \ref{gen-ext}} and {\it \ref{B[X,Y]}}. 

\subsection{General Extensions}\label{gen-ext}
Let $B$ be any integral $k$-domain, let $D\in{\rm LND}(B)$, $D\ne 0$, and let $A=\krn D$.
Let $\mathcal{F}_n\subset B$ be the degree modules and $I_n=D^n(\mathcal{F}_n)\subset A$ the image ideals for $D$, $n\ge 0$. 
Fix nonzero $a\in A$. Let $\mathcal{G}_n\subset\mathcal{F}_n$ be the $A$-submodule $\mathcal{G}_n=\sum_{0\le i\le n}a^i\mathcal{F}_i$ and let $R\subset B$ be the subalgebra generated by:
\[
\bigcup_{n\ge 0}\mathcal{G}_n=\sum_{n\ge 0}a^n\mathcal{F}_n
\]
Then $D$ restricts to $R$ and the degree modules of $D|_R$ are precisely $\mathcal{G}_n$, $n\ge 0$. Likewise, the image ideals for $D|_R$ are $a^nI_n$, $n\ge 0$. 

Let $B'=B[u]\cong B^{[1]}$ and extend $D$ to $D'$ on $B'$ by defining $D'u=a$. By {\it Lemma\,\ref{princ-6}} we have $[D',\partial_u]=0$. 
Let $K=\krn D'$ and $J=K\cap D'B'$. By {\it Lemma\,\ref{plinth}}, $\partial_u$ restricts to $K$ and $\partial_uJ\subset J$. Note that $K\cap B=A$. 
Let $\mathcal{F}_n^{\prime}\subset K$, $n\ge 0$, be the degree modules for $\partial_u$ restricted to $K$. 

Let $D_0$ be the extension of $D$ to $B[u]$ defined by $D_0u=0$ and let $\alpha :B_a[u]\to B_a[u]$ be the $k$-algebra automorphism $\alpha =\exp(-\frac{u}{a}D_0)$. 
Let $\varepsilon :B[u]\to B$ be evaluation at $u=0$. Since $uB[u]\cap K=(0)$ we see that $\varepsilon$ is injective on $K$. 

\begin{proposition}\label{main2} Assume that $aA\cap I_n=aI_n$ for each $n\ge 0$. 
\begin{itemize}
\item [{\bf (a)}] $\alpha\vert_R$ is an isomorphism of $R$ with $K$ and $\varepsilon\alpha\vert_R =id_R$. 
\medskip
\item [{\bf (b)}] If $B=k[\mathcal{F}_N]$ for some $N\ge 1$, then $R=k[\mathcal{G}_N]$ and $K=k[\mathcal{F}_N^{\prime}]$.
\medskip
\item [{\bf (c)}] If $B$ is normal then $R$ is normal. 
\medskip
\item [{\bf (d)}] If $B$ is a UFD then $R$ is a UFD. 
\end{itemize}
 \end{proposition}

Two preliminary lemmas are needed to prove the proposition. 

Let $\pi_u :B_a^{\prime}\to K_a$ be the Dixmier map for $D'$.
Given $n\ge 0$, define the mapping:
\[
\varphi_n:\mathcal{F}_n\to\mathcal{F}_n^{\prime}\,\, ,\,\, \varphi_n(f)=a^n\pi_u(f) 
\]
\begin{lemma}\label{phin} Let $m,n\ge 0$ be given. 
\begin{itemize}
\item [{\bf (a)}] $\varphi_n$ is an injective map of $A$-modules for each $n\ge 0$. 
\medskip
\item [{\bf (b)}] $\varphi_m(f)\varphi_n(g)=\varphi_{m+n}(fg)$ for each $m,n\ge 0$, $f\in\mathcal{F}_m$ and $g\in\mathcal{F}_n$. 
\medskip
\item [{\bf (c)}] $\varphi_n(p(0))=a^np(u)$ for each $p(u)\in K$.
\medskip
\item  [{\bf (d)}] $\deg_u\varphi_n(f)=\deg_D(f)$ for each $f\in\mathcal{F}_n$. 
\end{itemize}
\end{lemma}

\begin{proof} 
(a) Let $c\in A$ and $f,g\in\mathcal{F}_n$. Then 
\[
\textstyle\varphi_n(cg)=a^n\pi_u(cf)=a^n\pi_u(c)\pi_u(f)=a^nc\pi_u(f)=c\varphi_n(f)
\]
and:
\[
\textstyle\varphi_n(f+g)=a^n\pi_u(f+g)=a^n(\pi_u(f)+\pi_u(g))=a^n\pi_u(f)+a^n\pi_u(g)=\varphi_n(f)+\varphi_n(g)
\]
So $\varphi_n$ is an $A$-module homomorphism. 
If $\varphi_n(f)=0$ then $f\in\krn\pi_u=uB_a^{\prime}\cap K_a$. If $f\ne 0$ then $u\in K$, a contradiction. Therefore, $\krn\varphi_n=\{ 0\}$ and $\varphi_n$ is injective. 
\medskip

(b) $\varphi_m(f)\varphi_n(g)=a^m\pi_u(f)a^n\pi_u(g)=a^{m+n}\pi_u(fg)=\varphi_{m+n}(fg)$
\medskip

(c) Since the constant term of $\varphi_n(p(0))$ equals $a^np(0)$, 
we see that:
\[
\textstyle\varphi_n(p(0))-a^np(u)\in uB'\cap K
\]\
 If $\varphi_n(p(0))-a^np(u)\ne 0$ then $u\in K$, a contradiction. 
\medskip

(d) Given $p(u)\in\mathcal{F}_n^{\prime}$, if $g(u)=\varphi_n(p(0))$ then $g(u)=a^np(u)$ by part (c), and therefore:
\[
n=\deg_D(p(0))=\deg_ug(u)=\deg_up(u)
\]
Given $f\in\mathcal{F}_n$, if $p(u)=\varphi_n(f)$ then $a^nf=p(0)$ and $n=\deg_D(f)=\deg_Dp(0)=\deg_up(u)$. 
\end{proof}

\begin{lemma}\label{fnprime} Suppose that $aA\cap I_n=aI_n$ for each $n\ge 0$. Then for each $n\ge 0$:
\[
\mathcal{F}_n^{\prime}=\sum_{i=0}^n\varphi_i(\mathcal{F}_i)
\]
\end{lemma}

\begin{proof} Let $p:B[u]\to B[u]/aB[u]$ be natural surjection and let $p(x)=\bar{x}$ for $x\in B[u]$. Then:
\[
B[u]/aB[u]\cong (B/aB)[\bar{u}]\cong (B/aB)^{[1]}
\]
We proceed by induction on $n$. If $n=0$, then $\mathcal{F}_0^{\prime}=A=\varphi_0(\mathcal{F}_0)$. This gives the basis for induction. 

Assume that $\mathcal{F}_n^{\prime}=\sum_{i=0}^n\varphi_i(\mathcal{F}_i)$ for some $n\ge 0$. 
Define:
\[
M=\sum_{i=0}^{n+1}\varphi_i(\mathcal{F}_i)=\mathcal{F}_n^{\prime}+\varphi_{n+1}(\mathcal{F}_{n+1})\subset\mathcal{F}_{n+1}^{\prime}
\]
Suppose that $f\in a\mathcal{F}_{n+1}^{\prime}\cap M$ and write $f=g+h$, where $g\in \mathcal{F}_n^{\prime}$ and $h\in \varphi_{n+1}(\mathcal{F}_{n+1})$. 
Then $0=\bar{g}+\bar{h}$. On the one hand, $\deg_{\bar{u}}(\bar{g})\le n$, and on the other hand, 
$\bar{h}=\bar{y}\bar{u}^{n+1}$ for some $y\in I_{n+1}$. Therefore, $\bar{g}=\bar{h}=0$, which implies $g\in a\mathcal{F}_n^{\prime}$ and $y\in aA\cap I_{n+1}=aI_{n+1}$. 
So there exists $b\in\mathcal{F}_{n+1}$ such that $y=\frac{1}{(n+1)!}aD^{n+1}b$. It follows that:
\[ 
\deg_u(f-a\varphi_{n+1}(b))\le n \implies f-a\varphi_{n+1}(b)\in a\mathcal{F}_n^{\prime} \implies f\in a\left(\mathcal{F}_n^{\prime}+\varphi_{n+1}(\mathcal{F}_{n+1})\right)=aM
\]
Therefore, $a\mathcal{F}_{n+1}^{\prime}\cap M=aM$. By \cite{Freudenburg.17}, Theorem 8.9, it follows that $M=\mathcal{F}_{n+1}^{\prime}$. 
\end{proof} 

We can now give the proof of {\it Proposition\,\ref{main2}}. 
\begin{proof} Given $n\ge 0$, let $g=\sum_{i=0}^na^ig_i\in\mathcal{G}_n$ where $g_i\in\mathcal{F}_i$. 
Since 
\[
\textstyle \pi_u(f)=\exp (-\frac{u}{a}D_0)(f) \quad {\rm when}\quad f\in B
\]
we see that 
$\alpha (g)=\sum_{i=0}^n\varphi_i(g_i)\in\mathcal{F}_n^{\prime}$. By {\it Lemma\,\ref{fnprime}}, every element of $\mathcal{F}_n^{\prime}$ is of this form. 
Therefore, $\alpha$ restricts to an $A$-module isomorphism of $\mathcal{G}_n$ with $\mathcal{F}_n^{\prime}$. It follows that $\alpha (R)=K$. 

Suppose that $B$ is normal (respectively, a UFD). Then $B[u]$ is normal (respectively, a UFD), and $K$, being the kernel of a locally nilpotent derivation of $B[u]$, is normal (respectively, a UFD). Since $R$ is isomorphic to $K$ by part (a), $R$ is normal (respectively, a UFD). 
\end{proof} 

\subsection{Finite generation of fundamental invariants}\label{B[X,Y]}

Finite generation of invariant rings for representations of $\G_a$ over $\C$ was proved in the nineteenth century by P. Gordan by showing that this ring of invariants is isomorphic to a ring of 
$SL_2(\C )$-invariants on a larger affine space. Our proof uses the same technique in the framework of fundamental pairs. 

\begin{proposition}\label{isomorphism} Let $B$ be an affine $k$-domain with fundamental pair $(D,U)$ and $A=\krn D$. Let $B[X,Y]\cong B^{[2]}$ with fundamental pair
$(D+X\partial_Y,U+Y\partial_X)$. Then $A\cong B[X,Y]^{SL_2(k)}$ and $A$ is finitely generated as a $k$-algebra.
\end{proposition}
To prove this theorem, let $\Omega =\krn U$ and $\mathcal{F}_n=\krn D^{n+1}$. Define
$(D',U')=(D+X\partial_Y,U+Y\partial_X)$ and:
\[
A'=\krn D'\,\, ,\,\, \Omega^{\prime}=\krn U' \,\, ,\,\, \mathcal{F}_n^{\prime}=\krn (D')^{n+1}\,\, (n\ge 0)
\]
In addition, let $(\tilde{D},\tilde{U})$ be the trivial extension of $(D,U)$ to $B[X]$.\footnote{Note that $(D',U')$ is {\it not} an extension of $(\tilde{D},\tilde{U})$.} Then:
\[
\tilde{A}:=\krn\tilde{D}=A[X]\,\, ,\,\, \tilde{\Omega}:=\krn\tilde{U}=\Omega[X] \,\, ,\,\, \tilde{\mathcal{F}}_n:=\krn\tilde{D}^{n+1}=\mathcal{F}_n[X] \,\, (n\ge 0)
\]
Given $n\ge 0$ define $\tilde{\mathcal{G}}_n=k[\tilde{\mathcal{F}}_0 + X\tilde{\mathcal{F}}_1 +\cdots +X^n\tilde{\mathcal{F}}_n]\subset B[X]$, and define
 the map of $\tilde{A}$-modules $\varphi_n :\tilde{\mathcal{F}}_n\to A'$ by $\varphi_n(f)=X^n\pi_Y(f)$. Let $\beta :\tilde{R}\to A'$ be the $k$-algebra isomorphism from
 {\it Theorem\,\ref{main2}(a)}. 
 
In  {\it Proposition\,\ref{isomorphism}}, the fact that $A\cong B[X,Y]^{SL_2(k)}$  is an immediate consequence of the following proposition. The fact that $A$ is finitely generated as a $k$-algebra then follows from the Finiteness Theorem for reductive groups. 

\begin{proposition}\label{finite-generation} Let $N\ge 1$ be such that $B=k[\mathcal{F}_N]$ and let $\tilde{R}=k[\tilde{\mathcal{G}}_N]$. 
\begin{itemize}
\item [{\bf (a)}] $A'=k[\mathcal{T}]$ where $\mathcal{T}=\sum_{0\le i\le N}\varphi_i(\tilde{\mathcal{F}}_i)$.
\smallskip
\item [{\bf (b)}] Define the map of $A_0$-modules $\psi :\Omega\to \tilde{R}$ by $\psi (\omega )=X^n\omega$ for $\omega\in\Omega_{-n}$. 
Then $\psi$ is an injective $k$-algebra homomorphism. 
\smallskip
\item [{\bf (c)}] $\beta\psi$ is a $k$-algebra isomorphism of $\Omega$ with $A_0^{\prime}$.
\end{itemize}
\end{proposition}

\begin{proof} Part (a) is implied by {\it Proposition\,\ref{main2}}. 

For part (b), injectivity of $\psi$ is clear from its definition. In addition, when $\omega\in\Omega_{-n}$ and $\omega'\in\Omega_{-m}$, then 
$(X^n\omega)(X^m\omega')=X^{m+n}\omega\omega'$ gives $\psi (\omega\omega')=\psi (\omega )\psi (\omega')$. Since $\Omega$ is generated by its homogeneous elements, $\psi$ is an algebra homomorphism. 

For part (c) we have that $\beta\psi :\Omega\to A'$ is an injective algebra homomorphism. It must be shown that its image is $A_0^{\prime}$. To this end, suppose that 
$f\in \tilde{\mathcal{F}}_n$ and $\varphi_n(f)\in A_0^{\prime}$. By {\it Lemma\,\ref{degrees}(b)} we have:
\[
\deg (X^nf)=0 \implies f\in B_{-n}^{\prime}\cap\tilde{\mathcal{F}}_n\subset B_{-n}^{\prime}\cap\mathcal{F}_n^{\prime}
=\Omega_{-n}^{\prime}\implies f\in B[X]\cap\Omega_{-n}^{\prime}
\]
Since $B[X]_{-n}=\sum_iX^iB_{-(n+i)}$ we see that $f=\sum_iX^ib_{-(n+i)}$ for $b_{-(n+i)}\in B_{-(n+i)}$. This implies:
\[
D^nf=\sum_iX^iDb_{-(n+i)}\in A_n^{\prime}\cap B[X]=A_n[X] \implies Db_{-(n+i)}\in A_n\cap B_{n-i}
\]
Therefore, $b_{-(n+i)}=0$ for $i\ne 0$ and $f=b_{-n}\in \Omega_{-n}^{\prime}\cap B=\Omega_{-n}$. 

It follows that $A_0^{\prime}\subseteq\beta\psi (\Omega)\subseteq A_0^{\prime}$. Therefore, $\beta\psi$ is an isomorphism. 
 \end{proof}
 
 \subsection{The Extension Theorem}

Let $B$ be an affine $k$-domain and let $(D,U)$ be a fundamental pair for $B$ with induced 
$\Z$-grading $B=\bigoplus_{i\in\Z}B_i$.
Let $A=\krn D$ and, for each $i\ge 0$, let $A_i=A\cap B_i$ and $\mathcal{F}_i=\krn D^{i+1}$, $i\ge 0$. 
Since $B$ is affine, there exists $N\ge 1$ such that $B=k[\mathcal{F}_N]$. 

\begin{theorem}\label{main3} {\bf (Extension Theorem)}
Let $B[u]\cong B^{[1]}$ and $a\in A_0\setminus\{ 0\}$. Let $T\in {\rm LND}(B[u])$ be the extension of $D$ defined by $Tu=a$ and let $K=\krn T$. 
\begin{itemize}
\item [{\bf (a)}]  $K\cong_kk[\mathcal{F}_0+a\mathcal{F}_1+\cdots +a^N\mathcal{F}_N]$. In particular, $K$ is $k$-affine. 
\smallskip
\item [{\bf (b)}] Let $\delta\in {\rm LND}(K)$ be the restriction of $d/du$ to $K$. Then $\delta$ is fundamental. 
\smallskip
\item [{\bf (c)}] Assume that $k$ is algebraically closed. The quotient morphism $\pi :{\rm Spec}(B)\times\A_k^1\to {\rm Spec}(K)$ induced by $T$
is surjective if and only if $a\in k^*$. 
\end{itemize}
\end{theorem}

\begin{proof} Part (a) follows from the Structure Theorem and {\it Proposition\,\ref{main2}}. 

Part (b) follows from {\it Corollary\,\ref{Ksl2}}, since the isomorphism $\alpha :R\to K$ transforms $a^{-1}D$ into $d/du$. 

For part (c), note that, if $a\in k^*$ then $T$ has a slice and $\pi$ is surjective. 
Consider the case $a\not\in k^*$. 
Let $\mathcal{G}_N=\mathcal{F}_0+a\mathcal{F}_1+\cdots +a^N\mathcal{F}_N$ and let $R=k[\mathcal{G}_N]\subseteq B$. 
Let $Z={\rm Spec}(B)$ and $X={\rm Spec}(R)$. 
Recall from {\it Proposition\,\ref{main2}} that the isomorphism $\alpha^{-1} :K\to R$ is the restriction of the evaluation map $B[u]\to B$ sending $u$ to 0. 
Let $\tau :X\to Y$ be the isomorphism induced by $\alpha^{-1}$
and let $p: Z\to X$ and $\sigma :Z\times\A^1_k\to Z$ be the morphisms induced by the inclusions $R\subset B$ and $B\subset B[u]$.
Then the quotient map $\pi :Z\times\A^1_k\to Y$ factors as $\pi=\tau p\sigma$. 

\medskip
\begin{center}
\begin{tikzpicture}[scale=3]
\node at (0,.8){$Z$};
\node at (0,0){$X$};
\node at (1,0){$Y$};
\node at (1,.8){$Z\times\A^1_k$};
\draw[thick,->](.2,0)--(.8,0);
\draw[thick,->](0,.6)--(0,.2);
\draw[thick,->>](.75,.8)--(.2,.8);
\draw[thick,->](1,.6)--(1,.2);
\node at (1.15,.4){$\pi$};
\node at (.5,-.15){$\tau$};
\node at (-.15,.4){$p$};
\node at (.5,.95){$\sigma$};
\node at (.5,.1){$\cong$};
\end{tikzpicture}
\end{center}
\medskip

Note that both $A$ and $A_0=B^{SL_2(k)}$ are $k$-affine. 
In addition, since $B$ is affine, each $A$-module $\mathcal{F}_d$ is finitely generated. 
Suppose that:
\begin{enumerate}
\item $A_0=k[h_1,\hdots ,h_m]$
\item $I_1=t_1A+\cdots +t_lA$ for $t_i\in A$
\item $\mathcal{F}_i=Af_{i1}+\cdots +Af_{in_i}$ ($1\le i\le n$)
\end{enumerate}
Then:
\[
R=k [\mathcal{G}_N]=k[t_r,h_s,a^if_{ij}\, |\, 1\le r\le l\, ,\, 1\le s\le m\, ,\, 1\le i\le N\, ,\, 1\le j\le n_i ]\subseteq B
\]
Since $k$ is algebraically closed and $a$ is not constant we can choose $(\zeta_1,\hdots ,\zeta_m)\in {\rm Spec}(A_0)\subset\A^m_k$ such that $a(\zeta_1,\hdots ,\zeta_m)=0$. 
Let $\xi\in X$ be a point belonging to the set defined by an $R$-ideal of the form 
\[
J=(t_r-\lambda_r, h_s-\zeta_s, a^if_{in_i}-\mu_{in_i}\, |\, 1\le r\le l\, ,\, 1\le s\le m\, ,\, 1\le i\le n\, ,\, 1\le j\le n_i) 
\]
where $\lambda_r,\mu_{in_i}\in k$ and at least one of the $\mu_{in_i}$ is not zero. Then the fiber of $p$ over $\xi$ is defined by the $B$-ideal $JB=(1)$, meaning that this fiber is empty. Therefore, $p$ is not surjective in this case, which implies that $\pi$ is not surjective. 
\end{proof}

\begin{example}\label{pre-Winkelmann} {\rm Let $B=k[x_0,x_1,x_2]\cong k^{[3]}$ with basic fundamental pair:
\[
(D,U)=(x_0\partial_{x_1}+x_1\partial_{x_2} , 2x_1\partial_{x_0}+2x_2\partial_{x_1}) 
\]
Let $A=\krn D$ and $A_0=A\cap\krn U$. We have $A=k[x_0,f]\cong k^{[2]}$ and $A_0=k[f]\cong k^{[2]}$ for $f=2x_0x_2-x_1^2$. In addition, $B=k[\mathcal{F}_2]$ where
$\mathcal{F}_1=A+Ax_1$ and $\mathcal{F}_2=A+Ax_1+Ax_2$. 
Given nonzero $p(t)\in k[t]\cong k^{[1]}$ let $a=p(f)\in A_0$ and form the subring:
\[
R=k[\mathcal{F}_0+a\mathcal{F}_1+a^2\mathcal{F}_2]=k[x_0,f,ax_1,a^2x_2]
\]
Since $2x_0(a^2x_2)-(ax_1)^2=a^2f$ we see that:
\[
R\cong k[X,Y,Z,T]/(2XZ-Y^2-p(T)^2T)\cong B[t]/(f-p(t)^2t)
\]
From {\it Corollary\,\ref{Ksl2}} and {\it Theorem\,\ref{main2}} we see that $R$ is a UFD with fundamental pair $(a^{-1}D ,aU)$. 
Note that $R$ is singular if $p(t)$ has a root in $k$; and if $p(t)$ is a nonzero constant then $R\cong k^{[3]}$. 
These rings define a family of factorial affine threefolds which are $SL_2(k)$-varieties equipped with an equivariant birational dominant morphism to $\A^3_k$. 
Compare to {\it Proposition\,\ref{threefold}}. 
}
\end{example}

%%%%%%%%%%%%%%%%%%%%%%%%%%%%%%%%%%%%%%%%%%%%%%%%%%%%%%%%%%%%%

\section{Free Extensions of fundamental derivations}\label{examples}
\subsection{Winkelmann's Examples} In \cite{Winkelmann.90}, Winkelmann gave the first examples of free $\G_a$-actions on complex affine space which are not translations. 
One of these starts with the $SL_2(\C )$-representation $V_2$ and (in our terminology) the induced basic fundamental pair 
\[
(x\partial_y +y\partial_z , 2z\partial_y+2y\partial_x)
\]
on $\C [x,y,z]$. Then 
$A_0=\C [f]$ for $f=2xz-y^2$. Let $B=\C [x,y,z,w]$ and extend $D$ to $D'\in {\rm LND}(B)$ by $D'w=f+1$. 
The induced $\G_a$-action on $\A_{\C}^4$ is fixed-point free. Winkelmann showed that the topological quotient of the action is not Hausdorff, so this cannot be a translation. 
In a second example, he used $V_1\oplus V_1$ with its generating invariant to get a $\G_a$-action on $\A_{\C}^5$ which is locally trivial but not a translation. The algebraic quotient of this action is smooth. 
Finston and Jaradat  \cite{Finston.Jaradat.17} used $V_3$ with its generating invariant to get a locally trivial $\G_a$-action on $\A_{\C}^5$ with singular algebraic quotient. 
Their calculations are very involved and indicate that the computing demands of the van den Essen Kernel Algorithm make this algorithm impractical for finding invariant rings for such examples.
The Extension Theorem was developed to deal with these types of $\G_a$-actions and gives a quick way to determine their invariant rings in terms of the degree modules of the defining representation. 

\subsection{Generalizing Winkelmann's Construction} 
For the remainder of this section, assume that $k$ is a field of characteristic zero. Wherever, part (c) of the Extension Theorem is used, the reader should also assume that $k$ is algebraically closed. 

To construct examples, we begin with a representation of $SL_2(k)$ on $kx_0\oplus\cdots\oplus kx_n\cong k^{n+1}$. 
The examples we consider below involve small values of $n$ since only for these do we have sufficiently detailed knowledge of the $SL_2(k)$-invariants and degree modules required by the Extension Theorem. 

Let $B=k[x_0,\hdots ,x_n]\cong k^{[n+1]}$ for $n\ge 2$, let $(D,U)$ be the linear fundamental pair for $B$ induced by the given representation, and let 
$B=\bigoplus_{i\in\Z}B_i$ be the induced grading of $B$. 
Choose $a\in A_0$ of the form $a=1+h$ where $h\in A_0\cap (DB)$. 
Define $D'\in {\rm LND}(B[u])$ by $D'b=Db$ for $b\in B$ and $D'u=a$. The induced $\G_a$-action on $\A_k^{n+2}$ is free, since the image of $D'$ generates the unit ideal in $B[u]$. 
Note that, by {\it Proposition\,\ref{trivial-ext}}, the extended action is not fundamental. 

In order to obtain a set of algebra generators for $K:=\krn D'$ from {\it Theorem\,\ref{main2}} we need the degree modules $\mathcal{F}_0,\hdots ,\mathcal{F}_n$ for $D$. 
Start with the $\N$-grading $A=\bigoplus_{i\in\N}A_i$. 
Given $d\ge 1$, {\it Theorem\,\ref{main1}} shows that $I_d=\bigoplus_{i\ge d}A_i$ and from this one finds a set of homogeneous ideal generators
$I_d=(f_1,\hdots ,f_r)$. 
We know from {\it Lemma\,\ref{degrees}} that $\deg_Uf_i=\deg f_i\ge d$ for each $i$, 
so we can easily calculate $g_i=U^df_i$ such that $D^dg_i=c_1\cdots c_df_i$ for $c_j\in k$ as in {\it Lemma\,\ref{updown}} and thus obtain:
\[
\mathcal{F}_d=Ag_1+\cdots +Ag_r+\mathcal{F}_{d-1}
\]
Having chosen $a\in A_0$ we get the submodules:
\[
\mathcal{G}_d=\sum_{0\le i\le d}a^i\mathcal{F}_i \quad  (d\ge 0)
\]
Since $B=k[\mathcal{F}_N]$ for some $N$, {\it Theorem\,\ref{main2}} implies that $K\cong k[\mathcal{G}_N]\subset B$. 

Let $\pi :\A_k^{n+2}\to Y={\rm Spec}(K)$ be the quotient morphism induced by the inclusion $K\subset B[u]$. 
As in the proof of {\it Theorem\,\ref{main3}}, let $X={\rm Spec}(k[\mathcal{G}_N])$ and let $\tau :X\to Y$ be the isomorphism induced by $\alpha^{-1}$. 
Let $p: \A_k^{n+1}\to Y$ and $\sigma :\A_k^{n+2}\to \A_k^{n+1}$ be the morphisms induced by the inclusions $k[\mathcal{G}_N]\subset B$ and $B\subset B[u]$.
Then $\pi=\tau p\sigma$ (see figure below). 

\medskip
\begin{center}
\begin{tikzpicture}[scale=3]
\node at (0,.8){$\A_k^{n+1}$};
\node at (0,0){$X$};
\node at (1,0){$Y$};
\node at (1,.8){$\A_k^{n+2}$};
\draw[thick,->](.2,0)--(.8,0);
\draw[thick,->](0,.6)--(0,.2);
\draw[thick,->>](.75,.8)--(.2,.8);
\draw[thick,->](1,.6)--(1,.2);
\node at (1.15,.4){$\pi$};
\node at (.5,-.15){$\tau$};
\node at (-.15,.4){$p$};
\node at (.5,.95){$\sigma$};
\node at (.5,.1){$\cong$};
\end{tikzpicture}
\end{center}
\medskip

Since $a\not\in k^*$, {\it Theorem\,\ref{main3}} shows that $p$ and $\pi$ are not surjective, so the $\G_a$-action on $\A_k^{n+2}$ cannot be globally trivial. 
However, the action can be locally trivial and in this case every nonempty fiber of $\pi$ is connected, meaning that every nonempty fiber of $p$ is connected. 

\subsection{Examples} The examples in this section use $T_1,T_2,T_3\in k[x_0,x_1,x_2,x_3]$ defined by:
\[
T_1 = x_0\,\, ,\,\,T_2 = 2x_0x_2-x_1^2  \,\, ,\,\, T_3= 3x_0^2x_3-3x_0x_1x_2+x_1^3 
\]
\subsubsection{$V_1$} In this case:
\begin{enumerate}
\item $B=k[x_0,x_1]\cong k^{[2]}$ and $\deg (x_0,x_1)=(1,-1)$.
\smallskip
\item $A=k[x_0]$, $A_d=kx_0^d$ and $I_d=x_0^dA$.
\smallskip
\item $\mathcal{F}_1=A\oplus Ax_1$
\end{enumerate}
Since $A_0=k$ the only choice for $a$ is $a=1$. So $\mathcal{G}_1=\mathcal{F}_1$ and $K\cong k[\mathcal{G}_1]=B$.
We obtain a free $\G_a$-action on $\A_k^3$, and it is easy to see that this action is equivalent to a translation. Indeed, by \cite{Daigle.Kaliman.09}, every free $\G_a$-action $\A_k^3$ is a translation in a suitable system of coordinates. 

\subsection{$V_2$} In this case:
\begin{enumerate}
\item $B=k[x_0,x_1,x_2]\cong k^{[3]}$ and $\deg (x_0,x_1,x_2)=(2,0,-2)$. 
\smallskip
\item $A=k[T_1,T_2]\cong k^{[2]}$ and $A_0=k[T_2]\cong k^{[1]}$.
\smallskip
\item Since $A$ is generated in degree 2, we need $I_1$ and $I_2$. From {\it Theorem\,\ref{main1}} we obtain $I_1=I_2=(x_0)$. Therefore,
$\mathcal{F}_1=A+Ax_1$ and $\mathcal{F}_2=\mathcal{F}_1+ Ax_2$.
\end{enumerate}
The simplest choice here is $a=1+T_2$ and this gives Winkelmann's first example (\cite{Winkelmann.90}, Lemma 8). 
We find that
\[
\mathcal{G}_2=A+a\mathcal{F}_1+a^2\mathcal{F}_2=A+Aax_1+Aa^2x_2
\]
which implies:
\begin{eqnarray*}
K &\cong& k[T_1,T_2,(1+T_2)x_1,(1+T_2)^2x_2] \\
&\cong& k[x,y,v,w]/(2xw-v^2-y(1+y)^2) \\
&\cong& B[t]/(a-(t^3+2t^2+t+1))
\end{eqnarray*}
Note that $X={\rm Spec}(k[\mathcal{G}_2])$ has a unique singular point\footnote{In Lemma 10, Winkelmann mistakenly refers to $Y={\rm Spec}(K)$ as a smooth cubic.} $(x,y,v,w)=(0,-1,0,0)$. 
The fiber of $p:\A_k^3\to X$ over the singular point  is $\{x_0=1-x_1^2=0\}$, the union of two disjoint lines.
Therefore, the $\G_a$-action is not locally trivial. In fact, the action is not even proper; see \cite{Freudenburg.17}, 3.8.4.

\subsection{$V_1\oplus V_1$ (Smooth Case)} In this case:
\begin{enumerate}
\item $B=k[x_0,x_1,y_0,y_1]\cong k^{[4]}$ and $\deg (x_0,x_1,y_0,y_1)=(1,-1,1,-1)$.
\smallskip
\item $A=k[x_0,y_0,P]\cong k^{[3]}$ where $P=x_0y_1-y_0x_1$ and $A_0=k[P]$. 
\smallskip
\item Since $A$ is generated in degree 1 we need $I_1$. From {\it Theorem\,\ref{main1}} we obtain $I_1=(x_0,y_0)$. Therefore,
$\mathcal{F}_1=A+Ax_1+Ay_1$.
\end{enumerate}
The simplest choice here is $a=1+P$ and this gives Winkelmann's second example (\cite{Winkelmann.90}, Section 2). 
Since $1\in I_1B[u]+(1+P)B[u]$ the induced action on $\A_k^5$ is locally trivial. We obtain
\[
\mathcal{G}_1=A+a(A+Ax_1+Ay_1)=A+Aax_1+Aay_1
\]
which implies:
\begin{eqnarray*}
K &\cong& k[x_0,y_0,P,(1+P)x_1,(1+P)y_1] \\
&\cong& k[x,y,z,v,w]/(xw-yv-z(1+z)) \\
&\cong& B[t]/(a-(t^2+t+1))
\end{eqnarray*}
In this case $K$ is smooth. 

\subsection{$V_1\oplus V_1$ (Singular Case)}\label{singular} In the previous example choose $m\ge 2$ and $a=(1+P)^m$. The extended action is again locally trivial since 
$x_0,y_0\in J$ and $P\in x_0B+y_0B$. We obtain:
\[
K=k[x_0,y_0,P,(1+P)^mx_1,(1+P)^my_1]\cong k[x,y,z,v,w]/(xw-yv-(1+z)^mz)
\]
Thus, $X={\rm Spec}(k[\mathcal{G}_1])$ has a unique singular point of order $m-1$ at $(x,y,z,v,w)=(0,0,-1,0,0)$. 

\subsection{$V_3$}\label{finston} In this case:
\begin{enumerate}
\item $B= k[x_0,x_1,x_2,x_3]\cong k^{[4]}$ and $\deg (x_0,x_1,x_2,x_3)=(3,1,-1,-3)$.
\smallskip
\item $A=k[T_1,T_2,T_3,H]$ where $T_1^2H=T_2^3+T_3^2$; and $A_0=k[H]$. In particular:
\[
H=9x_0^2x_3^2-18x_0x_1x_2x_3 + 8x_0x_2^3 + 6x_1^3x_3- 3x_1^2x_2^2
\]
\item Since $\deg (T_1,T_2,T_3)=(3,2,3)$ we need $I_1,I_2$ and $I_3$. From {\it Theorem\,\ref{main1}} we obtain:
\[
I_1=I_2=(T_1,T_2,T_3) \quad {\rm and}\quad I_3=(T_1,T_3)
\]
\item We find elements 
\begin{eqnarray*}
 P_1&:=& \textstyle\frac{1}{2} UT_2=3x_0x_3-x_1x_2 \\
 P_2&:=& \textstyle\frac{1}{4} U^2T_2=3x_1x_3-2x_2^2 \\
 Q_1&:=& \textstyle\frac{1}{3}UT_3=3x_0x_1x_3-4x_0x_2^2+x_1^2x_2 \\
 Q_2&:=& \textstyle\frac{1}{12}U^2T_3=3x_1^2x_3-3x_0x_2x_3-x_1x_2^2 \\ 
 Q_3&:=& \textstyle\frac{1}{36}U^3T_3=3x_1x_2x_3-3x_0x_3^2-\textstyle{\frac{4}{3}x_2^3}
\end{eqnarray*}
such that the $D$-cables 
\[
 x_3\to x_2\to x_1\to T_1 \,\, ,\quad P_2\to P_1\to T_2\,\, ,\quad Q_3\to Q_2\to Q_1\to T_3
\]
are $U$-cables with the arrows reversed. We thereby obtain:
\subitem $\mathcal{F}_1=A+Ax_1+AP_1+AQ_1$
\smallskip
\subitem $\mathcal{F}_2=\mathcal{F}_1+Ax_2+AP_2+AQ_2$ 
\smallskip
\subitem $\mathcal{F}_3=\mathcal{F}_2+Ax_3+AQ_3$
\end{enumerate}

The simplest choice here is $a=1+H$ and this gives the example of Finston and Jaradat studied in \cite{Finston.Jaradat.17}. 
Since $1\in I_1B[u]+(1+H)B[u]$ the induced action on $\A_k^5$ is locally trivial. 
We obtain:
\[
K\cong k[\mathcal{G}_3]=A[ax_1,aP_1,aQ_1,a^2x_2,a^2P_2,a^2Q_2,a^3x_3,a^3Q_3]
\]
This confirms the result of Finston and Jaradat who used the van den Essen algorithm (and {\it Singular}) to find that $K$ is generated by 12 elements over $\C$.
They also showed that $Y={\rm Spec}(K)$ is singular in at least two points, and that the ideal of relations for the kernel generators 
has 155 generators. This was the first example of a locally trivial $\G_a$-action on an affine space having a singular algebraic quotient. 

\subsection{$V_1\oplus V_2$} $B=k[x_0,x_1,y_0,y_1,y_2]\cong k^{[5]}$ and $\deg(x_0,x_1,y_0,y_1,y_2)=(1,-1,2,0,-2)$. 
\begin{enumerate}
\item $A=k[x,y,z,v,w]$ with $x^2v-yw+z^2=0$, where
\[
x=x_0\,\, ,\,\, y=y_0\,\, ,\,\, z=[x_1,y_1]_1^D \,\, ,\,\, v=[y_2,y_2]_2^D \,\, ,\,\, w=[y_2,x_1^2]_2^D
\]
and $A_0=k[v,w]\cong k^{[2]}$. This is easily confirmed by the van den Essen algorithm. In particular, $w=2x_0^2y_2-2x_0x_1y_1+y_0x_1^2$. 
\smallskip
\item Since $\deg (x,y,z)=(1,2,1)$ we need $I_1$ and $I_2$. From {\it Theorem\,\ref{main1}} we obtain
\[
I_1=(x_0, y_0, z) \quad {\rm and}\quad I_2=y_0A+I_1^2
\]
and therefore, if $l=2x_0y_2-x_1y_1$, then $Dl=z$ and:
\[
\mathcal{F}_1=A+Ax_1+ Ay_1+Al \quad {\rm and}\quad \mathcal{F}_2=\mathcal{F}_1+Ay_2
\]
\end{enumerate}
The natural choice here is $a=1+w$. Since $1\in I_1B[u]+(1+w)B[u]$, the induced action on $\A_k^6$ is locally trivial. We obtain:
\[
K\cong k[x_0,y_0,z,v,w,ax_1,ay_1,al,a^2y_2]
\]
Question: Is the ring $K$ singular in this case?

\subsection{$V_2\oplus V_2$} $B=k [x_0,x_1,x_2,y_0,y_1,y_2]\cong k^{[6]}$ and $\deg (x_0,x_1,x_2,y_0,y_1,y_2)=(2,0,-2,2,0,-2)$
\begin{enumerate} 
\item $A=k[x,y,z,t,v,w]$ with $x^2v+y^2t+z^2-2xyw=0$, where:
\[
x=x_0\,\, ,\,\, y=y_0\,\, ,\,\, z=[x_1,y_1]_1^D\,\, ,\,\, t=[x_2,x_2]_2^D\,\, ,\,\, v=[y_2,y_2]_2^D \,\, ,\,\, w=[x_2,y_2]_2^D
\]
See \cite{Freudenburg.17}, Section 6.3.4. 
\smallskip
\item $A_0=k[t,v,w]\cong k^{[3]}$ and $A_{2d}$ is the vector space of $d$-forms in $x,y,z$ for $d\ge 1$. Since $A$ is generated in degree 2, we need $I_1$ and $I_2$. 
From {\it Theorem\,\ref{main1}} we obtain:
\[
I_1=I_2=(x_0,y_0,x_0y_1-y_0x_1)
\]
Therefore:
\[
\mathcal{F}_1=A+Ax_1+Ay_1+A(x_0y_2-y_0x_2) \quad {\rm and}\quad \mathcal{F}_2=\mathcal{F}_1+Ax_2+Ay_2+A(x_1y_2-y_1x_2)
\]
\end{enumerate}
A natural choice is $a=1+w=1+y_0x_2-x_1y_1+x_0y_2$, since choosing $1+t$ or $1+v$ just extends $V_2$. 
We obtain
\[
K\cong A[ax_1,ay_1,a(x_0y_2-y_0x_2), a^2x_2,a^2y_2,a^2(x_1y_2-y_1x_2)]
\]
which has 12 generators over $k$. The fiber of $p$ over the point in $X\subset\A_k^{12}$ where 
\begin{center}
$t=v=1$, $w=-1$, and all other generators are 0
\end{center}
consists of two disjoint planes
$\{ (0,0,1,1,*,*)\}\cup\{ (0,0,-1,-1,*,*)\}$ in $\A_k^6$.
Therefore, the action is not locally trivial. 

Question: Is the ring $K$ singular in this case?

\subsection{$V_2\oplus V_2\oplus V_2$} $B=k[x_0,x_1,x_2,y_0,y_1,y_2,z_0,z_1,z_2]\cong k^{[9]}$ and $\deg (x_0,x_1,x_2)=(2,0,-2)$, etc. 

$A$ is generated by the following 13 elements; see \cite{Freudenburg.17}, 6.3.4
\begin{equation}\label{first-set}
x_0\, ,\, y_0\, ,\, z_0\, ,\, [x_1,y_1]_1^D\, ,\, [x_1,z_1]_1^D\, ,\, [y_1,z_1]_1^D
\end{equation}
and
\begin{equation}\label{second-set}
[x_2,x_2]_2^D\, ,\, [x_2,y_2]_2^D\, ,\, [x_2,z_2]_2^D\, ,\, [y_2,y_2]_2^D\, ,\, [y_2,z_2]_2^D\, ,\, [z_2,z_2]_2^D
\end{equation}
and
\begin{equation}\label{third-set}
\delta = \det\begin{pmatrix} x_2&y_2&z_2\cr x_1&y_1&z_1\cr x_0&y_0&z_0\end{pmatrix}
\end{equation}
$A_0$ is generated by elements listed in (\ref{second-set}) and (\ref{third-set}). Elements in (\ref{first-set}) are of degree at most 2 and 
$A_{2d}$ is the vector space of $d$-forms in these elements. Therefore, we need $I_1$ and $I_2$. From {\it Theorem\,\ref{main1}} we find that $I_1=I_2$ is generated by the elements in (\ref{first-set}). 
Therefore:
\begin{center}
$\mathcal{F}_1=A+Ax_1+Ay_1+Az_1+A(x_0y_2-y_0x_2)+A(x_0z_2-z_0y_2)+A(y_0z_2-z_0y_2)$\\
$\mathcal{F}_2=\mathcal{F}_1+Ax_2+Ay_2+Az_2+A(x_1y_2-y_1x_2)+A(x_1z_2-z_1y_2)+A(y_1z_2-z_1y_2)$\\
\end{center}
A natural choice here is $a=1+\delta$. Since $\delta\in x_0B+y_0B+z_0B$ we see that $1\in I_1B[u]+(1+\delta )B[u]$ and the induced action is locally trivial, in contrast to the actions extending
$V_2$ or $V_2\oplus V_2$ given above. We see also that $K$ is generated by 25 elements. 

Question: Is the ring $K$ singular in this case?

\subsection{$\bigoplus_{i=1}^mV_1$, $m\ge 1$} $B=k[x_1,y_1,\hdots ,x_m,y_m]\cong k^{[2m]}$ and $\deg(x_i,y_i)=(1,-1)$, where $Dy_i=x_i$.  
\begin{enumerate}
\item $A= k[x_i , z_{ij}\, |\, 1\le j<i\le m]$ where $z_{ij}=[y_j,y_i]_1^D=x_iy_j-x_jy_i$, with relations:
\[
x_iz_{kj}+x_kz_{ij}+x_jz_{ik}=0 \quad (1\le j<k<i\le m)
\]
See \cite{Freudenburg.17}, 6.3.4. Also, $A_0=k[ z_{ij}\, |\, 1\le j<i\le m]$. 
\smallskip
\item We therefore need $I_1$. From {\it Theorem\,\ref{main1}} we obtain $I_1=(x_1,\hdots ,x_m)$ so:
\[
\mathcal{F}_1=A+Ay_1+\cdots +Ay_m
\]
\end{enumerate}
A natural choice is $a=1+\sum_{ij}z_{ij}$. Since $z_{ij}\in x_1B+\cdots +x_mB$ for each $i,j$, we see that
\[
\textstyle 1\in I_1B[u]+(1+\sum z_{ij})B[u]
\]
so the induced action is locally trivial. We have
\[
K\cong k[x_i , z_{ij} , ay_i\, |\, 1\le i,j\le m]
\]
with relations $x_i(ay_j)-x_j(ay_i)=az_{ij}$ in addition to those given above for $A$. 

%%%%%%%%%%%%%%%%%%%%%%%%%%%%%%%%%%%%%%%%%%%%%%%%%%%%%%%%%%%%%%%%%%%%%%

\section{Further Directions}

\subsection{$SL_2(\C )$-vector bundles of rank two} In \cite{Schwarz.89}, Schwarz used the theory of $G$-vector bundles to give the first examples of nonlinearizable actions of 
complex reductive groups on $\C^n$ for $n\ge 4$.
For $G=SL_2(\C )$ these examples include a nontrivial $G$-vector bundle of rank 4 over the $G$-module $V_2$ where the action of $G$ on the total space $\C^7$ is nonlinearizable. 

\begin{question} 
{\rm For $G=SL_2(\C )$ is every $G$-vector bundle of rank two over a representation linearizable?}
\end{question}

\begin{question} 
{\rm Is every algebraic action of $SL_2(\C )$ on $\C^5$ or $\C^6$ linearizable?}
\end{question}

\subsection{Affine 4-space} 
It is not currently known whether Panyushev's theorem generalizes to all fields of characteristic zero. The fact that $\R^4$ admits a non-linearizable action of the circle group $S^1$ (a real form of the torus $\R^*$) 
suggests that counterexamples might exist for a real form of $SL_2(\R )$; see \cite{Freudenburg.Moser-Jauslin.04}. 

\begin{question} 
{\rm Is every algebraic action of $SL_2(k)$ on $\A_k^4$ linearizable? More generally, let $\Gamma$ be an algebraic group over $k$ which is a $k$-form of $SL_2(k)$.
Is every algebraic action of $\Gamma$ on $\A_k^4$ linearlizable?}
\end{question}

\subsection{Cylinder over the Russell cubic} Let $X$ be the Russell cubic threefold over $k$. 
 {\it Example\,\ref{Russell}} shows that $X$ is not an $SL_2(k)$-variety. 
It is an open question whether the cylinder $X\times\A_k^1$ is isomorphic to $\A_k^4$. We ask if a weaker property holds.

\begin{question}
{\rm Does $X\times\A_k^1$ admit a nontrivial action of $SL_2(k)$?}
\end{question}
\noindent Note that Dubouloz \cite{Dubouloz.09} showed that $ML(X\times\A_k^1)=k$.

\subsection{2-Cylinder over a rigid variety} Let $S$ be an affine $k$-domain with $ML(S)=S$. Such a ring is called {\bf rigid}. 
Makar-Limanov showed that, if $S[X]\cong S^{[1]}$, then $ML(S[X])=S$; see \cite{Freudenburg.17}. 
Therefore, $S[X]$ does not have a nontrivial $SL_2(k)$-action.

\begin{question}
{\rm Let $S$ be a rigid affine $k$-domain. 
Is every nontrivial fundamental pair for the ring $S[X,Y]\cong S^{[2]}$ conjugate to $(X\partial_Y , Y\partial_X)$ over $S$?}
\end{question} 
\noindent Of particular interest is the invariant ring of the icosahedron, $S=k[x,y,z]/(x^5+y^3+z^2)$, which is a rigid UFD.

\subsection{Factorial $SL_2(\C )$-threefolds} The classification of Gizatullin and Popov shows that the affine plane $\C^2$ is the only factorial affine surface over $\C$ which admits a nontrivial action of $SL_2(\C )$. However, there are smooth factorial $SL_2(\C )$-threefolds with trivial units other than $\C^3$, for example, $SL_2(\C )$ as an affine $\C$-variety. 

\begin{question}
{\rm What are the smooth factorial affine threefolds over $\C$ with trivial units admitting a nontrivial $SL_2(\C )$-action?}
\end{question}
\noindent {\it Proposition\,\ref{threefold}} gives a partial answer to this question. Investigations of $SL_2(\C )$-threefolds can be found in \cite{Arzhantsev.97b,Arzhantsev.97a,Kekebus.00,Popov.73b}. 
\medskip

\subsection{Fundamental pairs} Recall from {\it Definition\,\ref{fun-pair}} the two conditions defining a fundamental pair $(D,U)\in {\rm LND}(B)^2$ on a commutative $k$-domain $B$. 
\begin{enumerate}
\item $[D,[D,U]]=-2D$ and $[U,[D,U]]=2U$, and
\smallskip
\item $B=\sum_{d\in\Z}B_d$ where $B_d=\krn ([D,U]-dI)$.
\end{enumerate}

\begin{question} {\rm Suppose that $B$ is an affine $k$-domain. Does condition (1) imply condition (2) in this case?}
\end{question} 

%%%%%%%%%%%%%%%%%%%%%%%%%%%%%%%%%%%%%%%%%%%%%%%%%%%%%%%%%%%%%%%%%%%%%%%%%%%%%%

\paragraph{\bf Acknowledgment.} The author is pleased to acknowledge that the comments of the referees of this paper led to a number of improvements over the first version. The author thanks I. Arzhantsev for helpful information about the history of the problems studied in this paper, and J.B. Tymkew for suggesting part (c),(ii) of the Structure Theorem. 
\medskip

\paragraph{\bf Conflict of Interest Statement.} The author declares that he has no conflict of interest. 

%%%%%%%%%%%%%%%%%%%%%%%%%%%%%%%%%%%%%%%%%%%%%%%%%%%%%%%%%%%%%%%%%%%%%%%%%%%%%%

%\bibliography{bibfile}
%\bibliographystyle{amsplain}

\vspace{.2in}

\noindent \address{Department of Mathematics\\
Western Michigan University\\
Kalamazoo, Michigan 49008} \,\,USA\\
\email{gene.freudenburg@wmich.edu}

\end{document}